\documentclass[preprint,3p,times]{elsarticle}

\usepackage{amssymb}
\usepackage{amsmath}
\usepackage[all,cmtip]{xy}
\usepackage{latexsym}
\usepackage{amsthm}
\usepackage{color}
\usepackage{hyperref}
\usepackage{stmaryrd}
\usepackage{yhmath}
\usepackage{graphicx}
\usepackage[hyphenbreaks]{breakurl}
\input diagxy

\journal{}

\theoremstyle{plain}
  \newtheorem{thm}{Theorem}[subsection]
  \newtheorem{lem}[thm]{Lemma}
  \newtheorem{prop}[thm]{Proposition}
  \newtheorem{cor}[thm]{Corollary}
\theoremstyle{definition}
  \newtheorem{defn}[thm]{Definition}
  
  \newtheorem{exmp}[thm]{Example}
  \newtheorem{rem}[thm]{Remark}

\DeclareMathAlphabet{\mathcal}{OMS}{cmsy}{m}{n}

\DeclareMathOperator{\id}{id}
\DeclareMathOperator{\ob}{ob}

\makeatletter
\def\ps@pprintTitle{%
 \let\@oddhead\@empty
  \let\@evenhead\@empty
  \def\@oddfoot{\vbox{\hsize=\textwidth\footnotesize
  \vskip 8pt
  \copyright 2015. This manuscript version is made available under the CC-BY-NC-ND 4.0 license \url{http://creativecommons.org/licenses/by-nc-nd/4.0/}. The published version is available at \url{http://dx.doi.org/10.1016/j.fss.2015.12.019}.\\
  }}%
  \let\@evenfoot\@oddfoot}
\makeatother

\def\oto{{\bfig\morphism<180,0>[\mkern-4mu`\mkern-4mu;]\place(86,0)[\circ]\efig}}

\renewcommand{\phi}{\varphi}
\newcommand{\da}{\downarrow}

\newcommand{\ua}{\uparrow}
\newcommand{\ra}{\rightarrow}
\newcommand{\la}{\leftarrow}
\newcommand{\nra}{{\rightarrow\hspace*{-1ex}{\mapstochar}\hspace*{1.1ex}}}
\newcommand{\nla}{{\leftarrow\hspace*{-0.8ex}{\mapstochar}\hspace*{1ex}}}
\newcommand{\lra}{\longrightarrow}
\newcommand{\lda}{\swarrow}
\newcommand{\rda}{\searrow}
\newcommand{\bs}{\backslash}
\newcommand{\Lra}{\Longrightarrow}

\newcommand{\rat}{\!\rightarrowtail\!}

\newcommand{\bv}{\bigvee}
\newcommand{\bw}{\bigwedge}

\newcommand{\dv}{\dashv}

\newcommand{\nat}{\natural}

\newcommand{\al}{\alpha}
\newcommand{\be}{\beta}

\newcommand{\lam}{\lambda}

\newcommand{\CB}{\mathcal{B}}
\newcommand{\CC}{\mathcal{C}}
\newcommand{\CD}{\mathcal{D}}
\newcommand{\CE}{\mathcal{E}}

\newcommand{\CJ}{\mathcal{J}}

\newcommand{\CQ}{\mathcal{Q}}

\newcommand{\sC}{{\sf C}}
\newcommand{\sD}{{\sf D}}
\newcommand{\sF}{{\sf F}}

\newcommand{\sI}{{\sf I}}

\newcommand{\sP}{{\sf P}}

\newcommand{\sS}{{\sf S}}

\newcommand{\sU}{{\sf U}}

\newcommand{\sj}{{\sf j}}

\newcommand{\sy}{{\sf y}}
\newcommand{\sPd}{{\sP^{\dag}}}

\newcommand{\cl}{{\sf cl}}

\newcommand{\Cd}{\hat{\sC}}
\newcommand{\Ccl}{\Cd_{\cl}}

\newcommand{\Id}{\hat{\sI}}
\newcommand{\Icl}{\Id_{\cl}}

\newcommand{\bbA}{\mathbb{A}}

\newcommand{\bbX}{\mathbb{X}}
\newcommand{\bbY}{\mathbb{Y}}
\newcommand{\bbZ}{\mathbb{Z}}

\newcommand{\Cat}{{\bf Cat}}

\newcommand{\CatCls}{{\bf CatCls}}

\newcommand{\Cls}{{\bf Cls}}
\newcommand{\ClsDist}{{\bf ClsDist}}

\newcommand{\ClsRel}{{\bf ClsRel}}
\newcommand{\ClsCloRel}{{\bf ClsRel}_{\cl}}

\newcommand{\Dist}{{\bf Dist}}

\newcommand{\Rel}{{\bf Rel}}
\newcommand{\Set}{{\bf Set}}

\newcommand{\Sup}{{\bf Sup}}
\newcommand{\QCat}{\CQ\text{-}\Cat}

\newcommand{\QDist}{\CQ\text{-}\Dist}

\newcommand{\QCatCls}{\CQ\text{-}\CatCls}
\newcommand{\QCls}{\CQ\text{-}\Cls}
\newcommand{\QClsDist}{\CQ\text{-}\ClsDist}
\newcommand{\QClsCloDist}{(\QClsDist)_{\cl}}

\newcommand{\QClsRel}{\CQ\text{-}\ClsRel}
\newcommand{\QClsCloRel}{(\QClsRel)_{\cl}}

\newcommand{\QRel}{\CQ\text{-}\Rel}
\newcommand{\QSup}{\CQ\text{-}\Sup}

\newcommand{\co}{{\rm co}}

\newcommand{\op}{{\rm op}}

\newcommand{\tphi}{\widetilde{\phi}}

\newcommand{\tzeta}{\widetilde{\zeta}}

\newcommand{\teta}{\widetilde{\eta}}

\newcommand{\tc}{\widetilde{c}}
\newcommand{\oc}{\overline{c}}
\newcommand{\od}{\overline{d}}

\newcommand{\td}{\widetilde{d}}

\newcommand{\tf}{\widetilde{f}}

\newcommand{\PA}{\sP\bbA}

\newcommand{\PX}{\sP\bbX}
\newcommand{\PY}{\sP\bbY}
\newcommand{\PZ}{\sP\bbZ}
\newcommand{\PdX}{\sPd\bbX}
\newcommand{\PdY}{\sPd\bbY}

\newcommand{\PdA}{\sPd\bbA}

\newcommand{\syd}{\sy^{\dag}}

\newcommand{\DQ}{\CD Q}

\newcommand{\ftra}{f^{\triangleright}}
\newcommand{\ftla}{f^{\triangleleft}}
\newcommand{\gtra}{g^{\triangleright}}

\newcommand{\zetra}{\zeta^{\triangleright}}
\newcommand{\zetla}{\zeta^{\triangleleft}}
\newcommand{\etra}{\eta^{\triangleright}}

\numberwithin{equation}{section}

\begin{document}

\begin{frontmatter}



\title{$\mathcal{Q}$-closure spaces}


\author{Lili Shen}
\ead{shenlili@yorku.ca}

\address{Department of Mathematics and Statistics, York University, Toronto, Ontario M3J 1P3, Canada}

\begin{abstract}
For a small quantaloid $\mathcal{Q}$, a $\mathcal{Q}$-closure space is a small category enriched in $\mathcal{Q}$ equipped with a closure operator on its presheaf category. We investigate $\mathcal{Q}$-closure spaces systematically with specific attention paid to their morphisms and, as preordered fuzzy sets are a special kind of quantaloid-enriched categories, in particular fuzzy closure spaces on fuzzy sets are introduced as an example. By constructing continuous relations that naturally generalize continuous maps, it is shown (in the generality of the $\mathcal{Q}$-version) that the category of closure spaces and closed continuous relations is equivalent to the category of complete lattices and $\sup$-preserving maps.
\end{abstract}


\end{frontmatter}

\tableofcontents


\section{Introduction}

A closure space consists of a (crisp) set $X$ and a closure operator $c$ on the powerset of $X$; that is, a monotone map $c:{\bf 2}^X\lra{\bf 2}^X$ with respect to the inclusion order of subsets such that $A\subseteq c(A)$, $cc(A)=c(A)$ for all $A\subseteq X$. However, $c$ may not satisfy $c(\varnothing)=\varnothing$, $c(A)\cup c(B)=c(A\cup B)$ for all $A,B\subseteq X$ that are necessary to make itself a \emph{topological} closure operator. The category $\Cls$ has closure spaces as objects and continuous maps as morphisms, where a map $f:(Y,d)\lra(X,c)$ between closure spaces is continuous if
$$f^{\ra}d(B)\subseteq cf^{\ra}(B)$$
for all $B\subseteq Y$. It is not difficult to observe that the continuity of the map $f$ is completely determined by its cograph
$$f^{\nat}:X\oto Y,\quad f^{\nat}=\{(x,y)\in X\times Y\mid x=fy\}$$
which must satisfy
$$(f^{\nat})^*d(B)\subseteq c(f^{\nat})^*(B)$$
for all $B\subseteq Y$; the notion of \emph{continuous relations} then comes out naturally by replacing the cograph $f^{\nat}$ with a general relation $\zeta:X\oto Y$ (i.e., $\zeta\subseteq X\times Y$) satisfying
$$\zeta^*d(B)\subseteq c\zeta^*(B)$$
for all $B\subseteq Y$, where $\zeta^*$ is part of the \emph{Kan adjunction} $\zeta^*\dv\zeta_*$ induced by $\zeta$ \cite{Shen2013a}:
$$\zeta^*(B)=\{x\in X\mid \exists y\in B:\ (x,y)\in\zeta\}.$$
The category $\ClsRel$ of closure spaces and continuous relations admits a natural quotient category $\ClsCloRel$ of closure spaces and \emph{closed} continuous relations, where a continuous relation $\zeta:(X,c)\oto(Y,d)$ is \emph{closed} if
$$\tzeta y:=\{x\in X\mid (x,y)\in\zeta\}$$
is a closed subset of $(X,c)$ for all $y\in Y$. It will be shown that $\ClsCloRel$ is equivalent to the category $\Sup$ of complete lattices and $\sup$-preserving maps (Corollary \ref{ClsCloRel_Sup_equivalent}):
\begin{equation} \label{ClsCloRel_equiv_Sup}
\ClsCloRel\simeq\Sup.
\end{equation}

The construction stated above will be explored in a much more general setting in this paper for $\CQ$-closure spaces, where $\CQ$ is a small \emph{quantaloid}. The theory of quantaloid-enriched categories was initiated by Walters \cite{Walters1981}, established by Rosenthal \cite{Rosenthal1996} and mainly developed in Stubbe's works \cite{Stubbe2005,Stubbe2006}. Based on the fruitful results of quantaloid-enriched categories, recent works of H{\"o}hle-Kubiak \cite{Hohle2011} and Pu-Zhang \cite{Pu2012} have established the theory of \emph{preordered fuzzy sets} through categories enriched in a quantaloid $\DQ$ induced by a divisible unital quantale $Q$. The survey paper \cite{Stubbe2014} is particularly recommended as an overview of this theory for the readership of fuzzy logicians and fuzzy set theorists.

Given a small quantaloid $\CQ$, a $\CQ$-closure space \cite{Shen2013a} is a small $\CQ$-category (i.e., a small category enriched in $\CQ$) $\bbX$ equipped with a $\CQ$-closure operator $c:\PX\lra\PX$ on the presheaf $\CQ$-category of $\bbX$. Before presenting a general form of the categorical equivalence (\ref{ClsCloRel_equiv_Sup}), we investigate $\CQ$-closure spaces systematically with specific attention paid to their morphisms: continuous $\CQ$-functors, continuous $\CQ$-distributors and finally, closed continuous $\CQ$-distributors.

Without assuming a high level of expertise by the readers on quantaloids, we recall in Section \ref{QCat} the basic notions and techniques of quantaloid-enriched categories that will be employed later. Next, Section \ref{Qclosure_space_continuous_functor} is devoted to the study of the category $\QCatCls$ of $\CQ$-closure spaces and continuous $\CQ$-functors. Explicitly, a $\CQ$-functor $f:(\bbX,c)\lra(\bbY,d)$ between $\CQ$-closure spaces is \emph{continuous} if
$$f^{\ra}c\leq df^{\ra}:\PX\lra\PY$$
with respect to the pointwise underlying preorder of $\CQ$-categories. We also derive a conceptual definition of the specialization (pre)order in a general setting as specialization $\CQ$-categories, which has the potential to go far beyond its use in this paper (see Remark \ref{distributor_specialization}).

The main result of this paper is presented in Section \ref{Continuous_distributor}, where we extend continuous $\CQ$-functors to \emph{continuous $\CQ$-distributors} as morphisms of $\CQ$-closure spaces; that is, $\CQ$-distributors $\zeta:(\bbX,c)\oto(\bbY,d)$ between $\CQ$-closure spaces with
$$\zeta^*d\leq c\zeta^*:\PY\lra\PX.$$
The resulting category, $\QClsDist$, admits a quotient category $\QClsCloDist$ of $\CQ$-closure spaces and \emph{closed} continuous $\CQ$-distributors, where a continuous $\CQ$-distributor $\zeta:(\bbX,c)\oto(\bbY,d)$ is \emph{closed} if its transpose
$$\tzeta:\bbY\lra\PX$$
sends every object $y$ of $\bbY$ to a closed presheaf of $(\bbX,c)$. Although the assignment
\begin{equation} \label{Xc_cPX}
(\bbX,c)\mapsto\sC(\bbX,c)
\end{equation}
sending a $\CQ$-closure space $(\bbX,c)$ to the complete $\CQ$-category $\sC(\bbX,c)$ of closed presheaves only yields a left adjoint functor from $\QCatCls$ to the category $\QSup$ of complete $\CQ$-categories and $\sup$-preserving $\CQ$-functors (Theorem \ref{C_I_adjoint}), a little surprisingly, the same assignment (\ref{Xc_cPX}) on objects gives rise to an equivalence of categories (Theorem \ref{Ccl_Icl_equivalence})
\begin{equation} \label{QCatClsCloDist_equiv_QSup}
\QClsCloDist^{\op}\simeq\QSup,
\end{equation}
which reduces to the equivalence (\ref{ClsCloRel_equiv_Sup}) when $\CQ={\bf 2}$.

As an example of $\CQ$-closure spaces, the topic of Section \ref{Example_Fuzzy_closure_spaces} will be \emph{fuzzy closure spaces} defined on fuzzy sets. Once a divisible unital quantale $Q$ is chosen as the truth table for fuzzy sets, $\DQ$-categories precisely describe preordered fuzzy sets \cite{Hohle2011,Pu2012}, and fuzzy powersets of fuzzy sets are exactly $\DQ$-categories of presheaves on discrete $\DQ$-categories \cite{Shen2013b}; consequently, $\DQ$-closure spaces with discrete underlying $\DQ$-categories naturally characterize fuzzy closure spaces whose underlying sets are fuzzy sets. It should be noted that fuzzy closure spaces on \emph{crisp} sets discussed in \cite{Ghanim1989,Mashhour1985}, i.e., \emph{crisp} sets $X$ equipped with closure operators on $L^X$ \cite{Bvelohlavek2001a,Lai2009,Shen2013} for a lattice $L$, are different from what we introduce here: we are concerned with \emph{fuzzy} sets (instead of crisp sets) equipped with closure operators on their \emph{fuzzy powersets}. In fact, fuzzy topological spaces in most of the existing theories (see \cite{Chang1968,Hoehle1995,Hohle1999,Liu1998,Zhang2007} for instance) are defined as \emph{crisp} sets equipped with certain kinds of fuzzy topological structures, and the study of $\DQ$-closure spaces presents the first step towards the study of fuzzy topologies on fuzzy sets, which are also expected to be given by \emph{fuzzy} sets equipped with topological structures on their \emph{fuzzy powersets}.

\section{Quantaloid-enriched categories} \label{QCat}

As preparations for our discussion, we recall the basic concepts of quantaloid-enriched categories \cite{Heymans2010,Rosenthal1996,Shen2014,Stubbe2005,Stubbe2006,Stubbe2014} in this section and fix the notations.

\subsection{Quantaloids and $\CQ$-categories}

A \emph{quantaloid} is a category enriched in the symmetric monoidal closed category $\Sup$. Explicitly, a quantaloid $\CQ$ is a (possibly large) category with ordered small hom-sets, such that
\begin{itemize}
\item each hom-set $\CQ(p,q)$ $(p,q\in\ob\CQ)$ is a complete lattice, and
\item the composition $\circ$ of $\CQ$-arrows preserves componentwise joins, i.e.,
$$v\circ\Big(\bv_{i\in I} u_i\Big)=\bv_{i\in I}v\circ u_i,\quad\Big(\bv_{i\in I} v_i\Big)\circ u=\bv_{i\in I}v_i\circ u$$
for all $\CQ$-arrows $u,u_i:p\lra q$ and $v,v_i:q\lra r$ $(i\in I)$.
\end{itemize}

Given $\CQ$-arrows $u:p\lra q$, $v:q\lra r$, $w:p\lra r$, the corresponding adjoints induced by the compositions
\begin{align*}
-\circ u\dv -\lda u:&\ \CQ(p,r)\lra\CQ(q,r),\\
v\circ -\dv v\rda -:&\ \CQ(p,r)\lra\CQ(p,q)
\end{align*}
satisfy
$$v\circ u\leq w\iff v\leq w\lda u\iff u\leq v\rda w,$$
where the operations $\lda$, $\rda$ are called \emph{left} and \emph{right implications} in $\CQ$, respectively.

A \emph{subquantaloid} of a quantaloid $\CQ$ is exactly a subcategory of $\CQ$ that is closed under the inherited joins of $\CQ$-arrows. A homomorphism $f:\CQ\lra\CQ'$ between quantaloids is an ordinary functor between the underlying categories that preserves joins of arrows: for all $\CQ$-arrows $u_i:p\lra q$ $(i\in I)$,
$$f\Big(\bv_{i\in I}u_i\Big)=\bv_{i\in I}fu_i.$$
A homomorphism of quantaloids is \emph{full} (resp. \emph{faithful}, an \emph{equivalence} of quantaloids) if the underlying functor is full (resp. faithful, an equivalence of underlying categories).

From now on $\CQ$ always denotes a \emph{small} quantaloid with a set $\CQ_0:=\ob\CQ$ of objects and a set $\CQ_1$ of arrows. The top and bottom $\CQ$-arrow in $\CQ(p,q)$ are respectively $\top_{p,q}$ and $\bot_{p,q}$, and the identity $\CQ$-arrow on $q\in\CQ_0$ is $1_q$.

Given a (``base'') set $T$, a set $X$ equipped with a map $|\text{-}|:X\lra T$ is called a \emph{$T$-typed set}, where the value $|x|\in T$ is the \emph{type} of $x\in X$. A map $f:X\lra Y$ between $T$-typed sets is \emph{type-preserving} if $|x|=|fx|$ for all $x\in X$. $T$-typed sets and type-preserving maps constitute the slice category $\Set\da T$.

Given a small quantaloid $\CQ$ and taking $\CQ_0$ as the set of types, a \emph{$\CQ$-relation} (also \emph{$\CQ$-matrix}) \cite{Heymans2010} $\phi:X\oto Y$ between $\CQ_0$-typed sets is given by a family of $\CQ$-arrows $\phi(x,y)\in\CQ(|x|,|y|)$ ($x\in X$, $y\in Y$). With the pointwise local order inherited from $\CQ$
$$\phi\leq\phi':X\oto Y\iff\forall x,y\in X:\ \phi(x,y)\leq\phi'(x,y),$$
the category $\QRel$ of $\CQ_0$-typed sets and $\CQ$-relations constitute a (large) quantaloid in which
\begin{align*}
&\psi\circ\phi:X\oto Z,\quad(\psi\circ\phi)(x,z)=\bv_{y\in Y}\psi(y,z)\circ\phi(x,y),\\
&\xi\lda\phi:Y\oto Z,\quad(\xi\lda\phi)(y,z)=\bw_{x\in X}\xi(x,z)\lda\phi(x,y),\\
&\psi\rda\xi:X\oto Y,\quad(\psi\rda\xi)(x,y)=\bw_{z\in Z}\psi(y,z)\rda\xi(x,z)
\end{align*}
for $\CQ$-relations $\phi:X\oto Y$, $\psi:Y\oto Z$, $\xi:X\oto Z$, and
$$\id_X:X\oto X,\quad\id_X(x,y)=\begin{cases}
1_{|x|}, & \text{if}\ x=y,\\
\bot_{|x|,|y|}, & \text{else}
\end{cases}$$
gives the identity $\CQ$-relation on $X$.

\begin{rem}
$\CQ$-relations between $\CQ_0$-typed sets may be thought of as \emph{multi-typed} and \emph{multi-valued} relations: a $\CQ$-relation $\phi:X\oto Y$ may be decomposed into a family of $\CQ(p,q)$-valued relations $\phi_{p,q}:X_p\oto Y_q$ $(p,q\in\CQ_0)$, where $\phi_{p,q}$ is the restriction of $\phi$ on $X_p$, $Y_q$ which, respectively, consist of elements in $X$, $Y$ with types $p$, $q$.
\end{rem}

A $\CQ$-relation $\phi:X\oto X$ on a $\CQ_0$-typed set $X$ is \emph{reflexive} (resp. \emph{transitive}) if $\id_X\leq\phi$ (resp. $\phi\circ\phi\leq\phi$). A (small) \emph{$\CQ$-category} $\bbX=(X,\al)$ is given by a $\CQ_0$-typed set $X$ equipped with a reflexive and transitive $\CQ$-relation $\al:X\oto X$; that is,
\begin{itemize}
\item $1_{|x|}\leq\al(x,x)$, and
\item $\al(y,z)\circ\al(x,y)\leq\al(x,z)$
\end{itemize}
for all $x,y,z\in X$. For the simplicity of notations, we denote a $\CQ$-category by $\bbX$ and write $\bbX_0:=X$, $\bbX(x,y):=\al(x,y)$ for $x,y\in\bbX_0$ when there is no confusion\footnote{We still denote a $\CQ$-category explicitly by a pair $(X,\al)$ when it is necessary to eliminate possible confusion, especially for \emph{specialization $\CQ$-categories} defined in Subsection \ref{Specialization_QCat} and \emph{preordered fuzzy sets} in Section \ref{Example_Fuzzy_closure_spaces}.}. There is a natural underlying preorder on $\bbX_0$ given by
$$x\leq y\iff|x|=|y|=q\quad\text{and}\quad 1_q\leq\bbX(x,y),$$
and we write $x\cong y$ if $x\leq y$ and $y\leq x$. A $\CQ$-category $\bbX$ is \emph{separated} (or \emph{skeletal}) if its underlying preorder is a partial order; that is, $x\cong y$ implies $x=y$ for all $x,y\in\bbX_0$.

A \emph{$\CQ$-functor} (resp. \emph{fully faithful} $\CQ$-functor) $f:\bbX\lra\bbY$ between $\CQ$-categories is a type-preserving map $f:\bbX_0\lra\bbY_0$ with $\bbX(x,x')\leq\bbY(fx,fx')$ (resp. $\bbX(x,x')=\bbY(fx,fx')$) for all $x,x'\in\bbX_0$. With the pointwise (pre)order of $\CQ$-functors given by
$$f\leq g:\bbX\lra\bbY\iff\forall x\in\bbX_0:\ fx\leq gx\iff\forall x\in\bbX_0:\ 1_{|x|}\leq\bbY(fx,gx),$$
$\CQ$-categories and $\CQ$-functors constitute a 2-category $\QCat$. Bijective fully faithful $\CQ$-functors are exactly isomorphisms in $\QCat$.

A pair of $\CQ$-functors $f:\bbX\lra\bbY$, $g:\bbY\lra\bbX$ forms an adjunction $f\dv g$ in $\QCat$ if $\bbY(fx,y)=\bbX(x,gy)$ for all $x\in\bbX_0$, $y\in\bbY_0$; or equivalently,
$1_{\bbX}\leq gf$ and $fg\leq 1_{\bbY}$, where $1_{\bbX}$ and $1_{\bbY}$ respectively denote the identity $\CQ$-functors on $\bbX$ and $\bbY$.

\begin{exmp} \phantomsection \label{QCat_exmp}
\begin{itemize}
\item[\rm (1)] For each $\CQ_0$-typed set $X$, $(X,\id_X)$ is a \emph{discrete} $\CQ$-category. In this paper, a $\CQ_0$-typed set $X$ is always assumed to be a discrete $\CQ$-category.
\item[\rm (2)] For each $q\in\CQ_0$, $\{q\}$ is a discrete $\CQ$-category with only one object $q$, in which $|q|=q$ and $\{q\}(q,q)=1_q$.
\item[\rm (3)] A $\CQ$-category $\bbA$ is a (full) \emph{$\CQ$-subcategory} of $\bbX$ if $\bbA_0\subseteq\bbX_0$ and $\bbA(x,y)=\bbX(x,y)$ for all $x,y\in\bbA_0$. 
\end{itemize}
\end{exmp}

\subsection{Presheaves on $\CQ$-categories and completeness} \label{Presheaves}

A \emph{$\CQ$-distributor} $\phi:\bbX\oto\bbY$ between $\CQ$-categories is a $\CQ$-relation $\phi:\bbX_0\oto\bbY_0$ such that
\begin{equation} \label{distributor_def}
\bbY\circ\phi\circ\bbX=\phi;
\end{equation}
or equivalently,
$$\bbY(y,y')\circ\phi(x,y)\circ\bbX(x',x)\leq\phi(x',y')$$
for all $x,x'\in\bbX_0$, $y,y'\in\bbY_0$. $\CQ$-categories and $\CQ$-distributors constitute a (large) quantaloid $\QDist$ in which compositions and implications are calculated as in $\QRel$; the identity $\CQ$-distributor on each $\CQ$-category $\bbX$ is given by the hom-arrows $\bbX:\bbX\oto\bbX$. It is obvious that $\QRel$ is a full subquantaloid of $\QDist$.

Each $\CQ$-functor $f:\bbX\lra\bbY$ induces a pair of $\CQ$-distributors given by
\begin{align*}
&f_{\nat}:\bbX\oto\bbY,\quad f_{\nat}(x,y)=\bbY(fx,y),\\
&f^{\nat}:\bbY\oto\bbX,\quad f^{\nat}(y,x)=\bbY(y,fx),
\end{align*}
called respectively the \emph{graph} and \emph{cograph} of $f$, which form an adjunction $f_{\nat}\dv f^{\nat}$ in the 2-category $\QDist$, i.e., $\bbX\leq f^{\nat}\circ f_{\nat}$ and $f_{\nat}\circ f^{\nat}\leq\bbY$. Furthermore,
$$(-)_{\nat}:\QCat\lra(\QDist)^{\co},\quad(-)^{\nat}:\QCat\lra(\QDist)^{\op}$$
are both 2-functors, where ``$\co$'' refers to reversing order in hom-sets.

%

A \emph{presheaf} with type $q$ on a $\CQ$-category $\bbX$ is a $\CQ$-distributor $\mu:\bbX\oto\{q\}$ (see Example \ref{QCat_exmp}(2) for the definition of $\{q\}$). Presheaves on $\bbX$ constitute a $\CQ$-category $\PX$ with $$\PX(\mu,\mu'):=\mu'\lda\mu=\bw_{x\in\bbX_0}\mu'(x)\lda\mu(x)$$
for all $\mu,\mu'\in\PX$. Dually, the $\CQ$-category $\PdX$ of \emph{copresheaves} on $\bbX$ consists of $\CQ$-distributors $\lam:\{q\}\oto\bbX$ as objects with type $q$ $(q\in\CQ_0)$ and
$$\PdX(\lam,\lam'):=\lam'\rda\lam=\bw_{x\in\bbX_0}\lam'(x)\rda\lam(x)$$
for all $\lam,\lam'\in\PdX$. It is easy to see that $\PdX\cong(\PX^{\op})^{\op}$, where ``$\op$'' means the dual of $\CQ$-categories as explained in the following remark:

\begin{rem} \label{QCat_dual}
Dual notions arise everywhere in the theory of $\CQ$-categories. To make this clear, it is best to first explain the \emph{dual} of a $\CQ$-category. In general, the dual of a $\CQ$-relation $\phi:X\oto Y$, written as
$$\phi^{\op}:Y\oto X,\quad \phi^{\op}(y,x)=\phi(x,y)\in\CQ(|x|,|y|)=\CQ^{\op}(|y|,|x|),$$
is not a $\CQ$-relation, but rather a $\CQ^{\op}$-relation. Correspondingly, the dual of a $\CQ$-category $\bbX$ is a $\CQ^{\op}$-category\footnote{The terminologies adopted here are not exactly the same as in the references \cite{Stubbe2005,Stubbe2006,Stubbe2014}: Our $\CQ$-categories are exactly $\CQ^{\op}$-categories in the sense of Stubbe.}, given by $\bbX^{\op}_0=\bbX_0$ and $\bbX^{\op}(x,y)=\bbX(y,x)$ for all $x,y\in\bbX_0$; a $\CQ$-functor $f:\bbX\lra\bbY$ becomes a $\CQ^{\op}$-functor $f^{\op}:\bbX^{\op}\lra\bbY^{\op}$ with the same mapping on objects but $g^{\op}\leq f^{\op}$ whenever $f\leq g:\bbX\lra\bbY$. Briefly, there is a 2-isomorphism
\begin{equation} \label{Qop_Cat_iso}
(-)^{\op}:(\QCat)^{\co}\cong\CQ^{\op}\text{-}\Cat.
\end{equation}
\end{rem}

\begin{exmp}
Given $q\in\CQ_0$, a presheaf on the one-object $\CQ$-category $\{q\}$ (see Example \ref{QCat_exmp}(2)) is exactly a $\CQ$-arrow $u:q\lra|u|$ for some $|u|\in\CQ_0$, thus $\sP\{q\}$ consists of all $\CQ$-arrows with domain $q$ as objects and
$$\sP\{q\}(u,u')=u'\lda u:p\lra p'$$
for all $\CQ$-arrows $u:q\lra p$, $u':q\lra p'$. Dually, $\sPd\{q\}$ is the $\CQ$-category of all $\CQ$-arrows with codomain $q$ and
$$\sPd\{q\}(v,v')=v'\rda v:r\lra r'$$
for all $\CQ$-arrows $v:r\lra q$, $v':r'\lra q$.
\end{exmp}

Each $\CQ$-distributor $\phi:\bbX\oto\bbY$ induces a \emph{Kan adjunction} \cite{Shen2013a} $\phi^*\dv\phi_*$ in $\QCat$ given by
\begin{align*}
&\phi^*:\PY\lra\PX,\quad \lam\mapsto\lam\circ\phi,\\
&\phi_*:\PX\lra\PY,\quad \mu\mapsto\mu\lda\phi
\end{align*}
and a \emph{dual Kan adjunction} \cite{Shen2014} $\phi_{\dag}\dv\phi^{\dag}$ given by
\begin{align*}
&\phi_{\dag}:\PdY\lra\PdX,\quad\lam\mapsto\phi\rda\lam,\\
&\phi^{\dag}:\PdX\lra\PdY,\quad\mu\mapsto\phi\circ\mu.
\end{align*}

\begin{prop} {\rm\cite{Heymans2010}} \label{star_graph_adjoint}
$(-)^*:(\QDist)^{\op}\lra\QCat$ and $(-)^{\dag}:(\QDist)^{\co}\lra\QCat$ are both 2-functorial, and one has two pairs of adjoint 2-functors
$$(-)^{\nat}\dv(-)^*:(\QDist)^{\op}\lra\QCat\quad\text{and}\quad(-)_{\nat}\dv(-)^{\dag}:(\QDist)^{\co}\lra\QCat.$$
\end{prop}

One may form several compositions out of the 2-functors in Proposition \ref{star_graph_adjoint}:
\begin{align*}
&(-)^{\ra}:=(\QCat\to^{(-)^{\nat}}(\QDist)^{\op}\to^{(-)^*}\QCat),\\
&(-)^{\la}:=((\QCat)^{\co\op}\to^{(-)_{\nat}^{\co\op}}(\QDist)^{\op}\to^{(-)^*}\QCat),\\
&(-)^{\nra}:=(\QCat\to^{(-)_{\nat}}(\QDist)^{\co}\to^{(-)^{\dag}}\QCat),\\
&(-)^{\nla}:=((\QCat)^{\co\op}\to^{(-)^{\nat\,\co\op}}(\QDist)^{\co}\to^{(-)^{\dag}}\QCat).
\end{align*}
Explicitly, each $\CQ$-functor $f:\bbX\lra\bbY$ gives rise to four $\CQ$-functors between the $\CQ$-categories of presheaves and copresheaves on $\bbX$, $\bbY$:
$$\begin{array}{llll}
f^{\ra}:=(f^{\nat})^*:&\PX\lra\PY,& f^{\la}:=(f_{\nat})^*=(f^{\nat})_*:&\PY\lra\PX,\\
f^{\nra}:=(f_{\nat})^{\dag}:&\PdX\lra\PdY,& f^{\nla}:=(f^{\nat})^{\dag}=(f_{\nat})_{\dag}:&\PdY\lra\PdX,
\end{array}$$
where $(f_{\nat})^*=(f^{\nat})_*$ and $(f^{\nat})^{\dag}=(f_{\nat})_{\dag}$ since one may easily verify
$$\lam\circ f_{\nat}=\lam\lda f^{\nat}\quad\text{and}\quad f^{\nat}\circ\lam'=f_{\nat}\rda\lam'$$
for all $\lam\in\PY$, $\lam'\in\PdY$ by routine calculation. As special cases of (dual) Kan adjunctions one immediately has
$$f^{\ra}\dv f^{\la}\quad\text{and}\quad f^{\nla}\dv f^{\nra}$$
in $\QCat$. Moreover, it is not difficult to obtain that
\begin{align}
(f^{\la}\lam)(x)=\lam(fx):|x|\lra|\lam|,\label{fla_lam}\\
(f^{\nla}\lam')(x)=\lam'(fx):|\lam'|\lra|x| \label{fnla_lam}
\end{align}
for all $\lam\in\PY$, $\lam'\in\PdY$ and $x\in\bbX_0$.


The following propositions are useful in the sequel and the readers may easily check their validity:

\begin{prop} {\rm\cite{Shen2014}} \label{fully_faithful_graph}
For each $\CQ$-functor $f:\bbX\lra\bbY$, the following statements are equivalent:
\begin{itemize}
\item[\rm (i)] $f$ is fully faithful.
\item[\rm (ii)] $f^{\nat}\circ f_{\nat}=\bbX$.
\item[\rm (iii)] $f^{\la}f^{\ra}=1_{\PX}$.
\item[\rm (iv)] $f^{\nla}f^{\nra}=1_{\PdX}$.
\end{itemize}
\end{prop}

\begin{prop} {\rm\cite{Shen2014}} \label{Yoneda_natural}
$\{\sy_{\bbX}:\bbX\lra\PX\mid\bbX\in\ob(\QCat)\}$ and $\{\syd_{\bbX}:\bbX\lra\PdX\mid\bbX\in\ob(\QCat)\}$ are respectively 2-natural transformations from the identity 2-functor on $\QCat$ to $(-)^{\ra}$ and $(-)^{\nra}$; that is, the diagrams
$$\bfig
\square<700,500>[\bbX`\bbY`\PX`\PY;f`\sy_{\bbX}`\sy_{\bbY}`f^{\ra}]
\square(1700,0)<700,500>[\bbX`\bbY`\PdX`\PdY;f`\syd_{\bbX}`\syd_{\bbY}`f^{\nra}]
\efig$$
commute for all $\CQ$-functors $f:\bbX\lra\bbY$.
\end{prop}

A $\CQ$-category $\bbX$ is \emph{complete} if each $\mu\in\PX$ has a \emph{supremum} $\sup\mu\in\bbX_0$ of type $|\mu|$ such that
$$\bbX(\sup\mu,-)=\bbX\lda\mu;$$
or equivalently, the \emph{Yoneda embedding} $\sy:\bbX\lra\PX,\ x\mapsto\bbX(-,x)$ has a left adjoint $\sup:\PX\lra\bbX$ in $\QCat$. It is well known that $\bbX$ is a complete $\CQ$-category if and only if $\bbX^{\op}$ is a complete $\CQ^{\op}$-category \cite{Stubbe2005}, where the completeness of $\bbX^{\op}$ may be translated as each $\lam\in\PdX$ admitting an \emph{infimum} $\inf\lam\in\bbX_0$ of type $|\lam|$ such that
$$\bbX(-,\inf\lam)=\lam\rda\bbX;$$
or equivalently, the \emph{co-Yoneda embedding} $\syd:\bbX\oto\PdX,\ x\mapsto\bbX(x,-)$ admitting a right adjoint $\inf:\PdX\lra\bbX$ in $\QCat$.

\begin{lem}[Yoneda] {\rm\cite{Stubbe2005}} \label{Yoneda_lemma}
Let $\bbX$ be a $\CQ$-category and $\mu\in\PX$, $\lam\in\PdX$. Then
$$\mu=\PX(\sy-,\mu)=\sy_{\nat}(-,\mu),\quad\lam=\PdX(\lam,\syd-)=(\syd)^{\nat}(\lam,-).$$
In particular, both $\sy$ and $\syd$ are fully faithful $\CQ$-functors.
\end{lem}

In a $\CQ$-category $\bbX$, the \emph{tensor} of a $\CQ$-arrow $u:|x|\lra q$ and $x\in\bbX_0$, denoted by $u\otimes x$, is an object in $\bbX_0$ of type $|u\otimes x|=q$ such that $$\bbX(u\otimes x,-)=\bbX(x,-)\lda u.$$
$\bbX$ is \emph{tensored} if $u\otimes x$ exists for all $x\in\bbX_0$ and $\CQ$-arrows $u\in\sP\{|x|\}$. Dually, $\bbX$ is \emph{cotensored} if $\bbX^{\op}$ is a tensored $\CQ^{\op}$-category. Explicitly, the \emph{cotensor} of a $\CQ$-arrow $v:q\lra|x|$ and $x\in\bbX_0$ is an object $v\rat x\in\bbX_0$ of type $q$ satisfying
$$\bbX(-,v\rat x)=v\rda\bbX(-,x).$$

A $\CQ$-category $\bbX$ is \emph{order-complete} if each $\bbX_q$, the $\CQ$-subcategory of $\bbX$ with all the objects of type $q\in\CQ_0$, admits all joins (or equivalently, all meets) in the underlying preorder.

\begin{thm} {\rm\cite{Stubbe2006}} \label{complete_tensor}
A $\CQ$-category $\bbX$ is complete if, and only if, $\bbX$ is tensored, cotensored and order-complete. In this case,
$$\sup\mu=\bv_{x\in\bbX_0}\mu(x)\otimes x,\quad\inf\lam=\bw_{x\in\bbX_0}\lam(x)\rat x$$
for all $\mu\in\PX$ and $\lam\in\PdX$, where $\bv$ and $\bw$ respectively denote the underlying joins and meets in $\bbX$; conversely,
$$u\otimes x=\sup(u\circ\sy x),\quad v\rat x=\inf(\syd x\circ v)$$
for all $x\in\bbX_0$ and $\CQ$-arrows $u\in\sP\{|x|\}$, $v\in\sPd\{|x|\}$, and
$$\bv_{i\in I}x_i=\sup\bv_{i\in I}\sy x_i,\quad\bw_{i\in I}x_i=\inf\bv_{i\in I}\syd x_i$$
for all $\{x_i\}_{i\in I}\subseteq\bbX_q$ $(q\in\CQ_0)$, where the $\bv$ and $\bw$ on the left hand sides respectively denote the underlying joins and meets in $\bbX$, and the $\bv$ on the right hand sides respectively denote the joins in $\QDist(\bbX,\{q\})$ and $\QDist(\{q\},\bbX)$.
\end{thm}

\begin{rem} \label{complete_nonempty}
A complete $\CQ$-category $\bbX$ has at least one object of type $q$ for every $q\in\CQ_0$, i.e., the bottom element in the underlying preorder of $\bbX_q$ as the empty join, since $\bbX$ is necessarily order-complete by Theorem \ref{complete_tensor}.
\end{rem}

\begin{exmp} \label{PX_tensor_complete} {\rm\cite{Stubbe2006}}
For each $\CQ$-category $\bbX$, $\PX$ and $\PdX$ are both separated, tensored, cotensored and complete $\CQ$-categories. It is easy to check that tensors and cotensors in $\PX$ are given by
$$u\otimes\mu=u\circ\mu,\quad v\rat\mu=v\rda\mu$$
for all $\mu\in\PX$ and $\CQ$-arrows $u\in\sP\{|\mu|\}$, $v\in\sPd\{|\mu|\}$, and consequently
$$\sup\Phi=\bv_{\mu\in\PX}\Phi(\mu)\circ\mu=\Phi\circ(\sy_{\bbX})_{\nat}=\sy_{\bbX}^{\la}\Phi,\quad\inf\Psi=\bw_{\mu\in\PX}\Psi(\mu)\rda\mu=\Psi\rda(\sy_{\bbX})_{\nat}$$
for all $\Phi\in\sP(\PX)$ and $\Psi\in\sPd(\PX)$, where we have applied the Yoneda lemma (Lemma \ref{Yoneda_lemma}) to get $\mu=(\sy_{\bbX})_{\nat}(-,\mu)$.
\end{exmp}

\begin{prop} {\rm\cite{Stubbe2006}} \label{la_tensor_preserving}
Let $f:\bbX\lra\bbY$ be a $\CQ$-functor, with $\bbX$ tensored (resp. cotensored). Then $f$ is a left (resp. right) adjoint in $\QCat$ if and only if
\begin{itemize}
\item[\rm (1)] $f$ preserves tensors (resp. cotensors) in the sense that $f(u\otimes_{\bbX}x)= u\otimes_{\bbY}fx$ (resp. $f(v\rat_{\bbX}x)=v\rat_{\bbY}fx$) for all $x\in\bbX_0$ and $\CQ$-arrows $u\in\sP\{|x|\}$ (resp. $v\in\sPd\{|x|\}$), and
\item[\rm (2)] $f$ is a left (resp. right) adjoint between the underlying preordered sets of $\bbX$ and $\bbY$.
\end{itemize}
\end{prop}

\begin{prop} {\rm\cite{Stubbe2005}} \label{la_sup_preserving}
Let $f:\bbX\lra\bbY$ be a $\CQ$-functor, with $\bbX$ complete. Then $f$ is a left (resp. right) adjoint in $\QCat$ if, and only if, $f$ is \emph{$\sup$-preserving} (resp. \emph{$\inf$-preserving}) in the sense that $f\sup_{\bbX}=\sup_{\bbY}f^{\ra}$ (resp. $f\inf_{\bbX}=\inf_{\bbY}f^{\nra}$).
\end{prop}

Therefore, left adjoint $\CQ$-functors between complete $\CQ$-categories are exactly $\sup$-preserving $\CQ$-functors. Separated complete $\CQ$-categories and $\sup$-preserving $\CQ$-functors constitute a 2-subcategory of $\QCat$ and we denote it by $\QSup$. In fact, it is not difficult to verify that $\QSup$ is a (large) quantaloid with the local order inherited from $\QCat$.

\subsection{$\CQ$-closure operators and $\CQ$-closure systems}

A $\CQ$-functor $c:\bbX\lra\bbX$ is a \emph{$\CQ$-closure operator} on $\bbX$ if $1_{\bbX}\leq c$ and $cc\leq c$, where the second condition actually becomes $cc\cong c$ since the reverse inequality already holds by the first condition. The most prominent example is that each pair $f\dv g:\bbY\lra\bbX$ of adjoint $\CQ$-functors induces a $\CQ$-closure operator $gf:\bbX\lra\bbX$.

A $\CQ$-subcategory $\bbA$ of $\bbX$ is a \emph{$\CQ$-closure system} if the inclusion $\CQ$-functor $j:\bbA\ \to/^(->/\bbX$ has a left adjoint in $\QCat$.

The dual notions are \emph{$\CQ$-interior operators} and \emph{$\CQ$-interior systems}, which correspond bijectively to $\CQ^{\op}$-closure operators and $\CQ^{\op}$-closure systems under the isomorphism (\ref{Qop_Cat_iso}) in Remark \ref{QCat_dual}.

\begin{rem}
In the language of category theory, $\CQ$-closure operators are exactly \emph{$\CQ$-monads} (note that the ``$\CQ$-natural transformation'' between $\CQ$-functors is simply given by the local order in $\QCat$), and $\CQ$-closure systems are precisely \emph{reflective} $\CQ$-subcategories.
\end{rem}


The following characterizations of $\CQ$-closure operators and $\CQ$-closure systems can be deduced from the similar results in \cite{Shen2014,Shen2013a}. Their direct proofs are also straightforward and will be left to the readers:

\begin{prop} \label{closure_operator_adjoint_inclusion}
Let $c:\bbX\lra\bbX$ be a $\CQ$-functor. Then $c$ is a $\CQ$-closure operator on $\bbX$ if, and only if, its codomain restriction $\oc:\bbX\lra\bbA$ is left adjoint to the inclusion $\CQ$-functor $j:\bbA\ \to/^(->/\bbX$, where
$$\bbA_0=\{cx\mid x\in\bbX_0\}.$$
In particular, $\bbA$ is a $\CQ$-closure system of $\bbX$.
\end{prop}


\begin{prop} \label{closure_system_condition}
Let $\bbA$ be a $\CQ$-subcategory of a separated complete $\CQ$-category $\bbX$ and $j:\bbA\ \to/^(->/\bbX$ the inclusion $\CQ$-functor. The following statements are equivalent:
\begin{itemize}
\item[\rm (i)] $\bbA$ is a $\CQ$-closure system of $\bbX$.
\item[\rm (ii)] $\bbA_0=\{cx\mid x\in\bbX_0\}$ for some $\CQ$-closure operator $c:\bbX\lra\bbX$.
\item[\rm (iii)] $\bbA$ is closed with respect to infima in $\bbX$ in the sense that ${\inf}_{\bbX}j^{\nra}\lam\in\bbA_0$ for all $\lam\in\PdA$.
\item[\rm (iv)] $\bbA$ is closed with respect to cotensors and underlying meets in $\bbX$.
\end{itemize}
\end{prop}

\begin{cor} \label{closure_system_sup}
Each $\CQ$-closure system $\bbA$ of a complete $\CQ$-category $\bbX$ is itself a complete $\CQ$-category. Furthermore, let $c:\bbX\lra\bbA$ be the left adjoint of the inclusion $\CQ$-functor $j:\bbA\ \to/^(->/\bbX$, one has
$${\sup}_{\bbA}\mu=c\cdot{\sup}_{\bbX}j^{\ra}\mu,\quad{\inf}_{\bbA}\lam={\inf}_{\bbX}j^{\nra}\lam$$
for all $\mu\in\PA$, $\lam\in\PdA$. In particular,
$$u\otimes_{\bbA}x=c(u\otimes_{\bbX}x),\quad v\rat_{\bbA}x=v\rat_{\bbX}x$$
for all $x\in\bbA_0$ and $\CQ$-arrows $u\in\sP\{|x|\}$, $v\in\sPd\{|x|\}$, and
$$\bigsqcup_{i\in I}x_i=c\Big(\bv_{i\in I}x_i\Big),\quad \bigsqcap_{i\in I}x_i=\bw_{i\in I}x_i$$
for all $\{x_i\}_{i\in I}\subseteq\bbA_0$, where we write $\bigsqcup$, $\bigsqcap$ respectively for the underlying joins, meets in $\bbA$, and $\bv$, $\bw$ for those in $\bbX$.
\end{cor}

The readers may easily write down the dual results of the above conclusions for $\CQ$-interior operators and $\CQ$-interior systems, which we skip here.

\section{$\CQ$-closure spaces and continuous $\CQ$-functors} \label{Qclosure_space_continuous_functor}

We introduce $\CQ$-closure spaces in this section and investigate the category of $\CQ$-closure spaces and continuous $\CQ$-functors as a natural extension of the well-known category $\Cls$ of closure spaces and continuous maps.

\subsection{The category of $\CQ$-closure spaces and continuous $\CQ$-functors} \label{QCatCls}

A \emph{$\CQ$-closure space} \cite{Shen2013a} is a pair $(\bbX,c)$ that consists of a $\CQ$-category $\bbX$ and a $\CQ$-closure operator $c:\PX\lra\PX$ on $\PX$. A \emph{continuous $\CQ$-functor} $f:(\bbX,c)\lra(\bbY,d)$ between $\CQ$-closure spaces is a $\CQ$-functor $f:\bbX\lra\bbY$ such that
$$f^{\ra}c\leq df^{\ra}:\PX\lra\PY.$$
$\CQ$-closure spaces and continuous $\CQ$-functors constitute a 2-category $\QCatCls$ with the local order inherited from $\QCat$.

In a $\CQ$-closure space $(\bbX,c)$, since $\PX$ is separated, $c$ is idempotent and the corresponding $\CQ$-closure system
$$\sC(\bbX,c):=\{c\mu\mid\mu\in\PX\}=\{\mu\in\PX\mid c\mu=\mu\}$$
consists of fixed points of $c$, which is a complete $\CQ$-category since so is $\PX$. A presheaf $\mu\in\PX$ is \emph{closed} if $\mu\in\sC(\bbX,c)$. It follows from Propositions \ref{closure_operator_adjoint_inclusion}, \ref{closure_system_condition} that $\CQ$-closure operators on $\PX$ correspond bijectively to $\CQ$-closure systems of $\PX$, thus a $\CQ$-closure space on $\bbX$ is completely determined by the $\CQ$-closure system of closed presheaves.

The following proposition shows that continuous $\CQ$-functors may be characterized as the inverse images of closed presheaves staying closed, and we will prove its generalized version in the next section (see Proposition \ref{continuous_dist_condition}):

\begin{prop} {\rm\cite{Shen2013a}} \label{continuous_functor_condition}
Let $(\bbX,c)$, $(\bbY,d)$ be $\CQ$-closure spaces and $f:\bbX\lra\bbY$ a $\CQ$-functor. The following statements are equivalent:
\begin{itemize}
\item[\rm (i)] $f:(\bbX,c)\lra(\bbY,d)$ is a continuous $\CQ$-functor.
\item[\rm (ii)] $df^{\ra}c\leq df^{\ra}$, thus $df^{\ra}c=df^{\ra}$.
\item[\rm (iii)] $cf^{\la}d\leq f^{\la}d$, thus $cf^{\la}d=f^{\la}d$.
\item[\rm (iv)] $f^{\la}\lam\in\sC(\bbX,c)$ whenever $\lam\in\sC(\bbY,d)$.
\end{itemize}
\end{prop}

$\QCatCls$ is a well-behaved category, which not only has all small colimits and small limits, but also possesses ``all possible'' large colimits and large limits that a locally small category can have; in fact, $\QCatCls$ is \emph{totally cocomplete} and \emph{totally complete} (will be explained below). To see this, we first establish its topologicity \cite{Adamek1990} over $\QCat$.

Recall that given a faithful functor $U:\CE\lra\CB$, a \emph{$U$-structured source} from $S\in\ob\CB$ is given by a (possibly large) family of objects $Y_i\in\ob\CE$ and $\CB$-morphisms $f_i:S\lra U Y_i$ ($i\in I$). A \emph{lifting} of $(f_i:S\lra UY_i)_{i\in I}$ is an $\CE$-object $X$ together with a family of $\CE$-morphisms $\check{f_i}:X\lra Y_i$ such that $UX=S$ and $U\check{f_i}=f_i$ for all $i\in I$, and the lifting is \emph{$U$-initial} if any $\CB$-morphism $g:UZ\lra S$ lifts to an $\CE$-morphism $\check{g}:Z\lra X$ as soon as every $\CB$-morphism $f_i g:UZ\lra UY_i$ lifts to an $\CE$-morphism $h_i:Z\lra Y_i$ $(i\in I)$. $U$ is called \emph{topological} if all $U$-structured sources admit $U$-initial liftings. It is well known that $U:\CE\lra\CB$ is topological if, and only if, $U^{\op}:\CE^{\op}\lra\CB^{\op}$ is topological (see \cite[Theorem 21.9]{Adamek1990}).
$$\bfig
\ptriangle/->`<-`<-/<800,500>[X`Y_i`Z;\check{f_i}`\check{g}`h_i]
\ptriangle(2000,0)/->`<-`<-/<800,500>[S`UY_i`UZ;f_i`g`Uh_i]
\morphism(1200,250)/|->/<400,0>[`;U]
\efig$$

\begin{prop} \label{U_topological}
The forgetful functor $\sU:\QCatCls\lra\QCat$ is topological.
\end{prop}

\begin{proof}
$\sU$ is obviously faithful. Given a (possibly large) family of $\CQ$-closure spaces $(\bbY_i,d_i)$ and $\CQ$-functors $f_i:\bbX\lra\bbY_i$ $(i\in I)$, we must find a $\CQ$-closure space $(\bbX,c)$ such that
\begin{itemize}
\item every $f_i:(\bbX,c)\lra(\bbY_i,d_i)$ is a continuous $\CQ$-functor, and
\item for every $\CQ$-closure space $(\bbZ,e)$, any $\CQ$-functor $g:\bbZ\lra\bbX$ becomes a continuous $\CQ$-functor $g:(\bbZ,e)\lra(\bbX,c)$ whenever all $f_i g:(\bbZ,e)\lra(\bbY_i,d_i)$ $(i\in I)$ are continuous $\CQ$-functors.
\end{itemize}
To this end, one simply defines $c=\displaystyle\bw_{i\in I}f_i^{\la}d_i f_i^{\ra}$, i.e., the meet of the composite $\CQ$-functors
$$\PX\to^{f_i^{\ra}}\PY_i\to^{d_i}\PY_i\to^{f_i^{\la}}\PX.$$
Then $c$ is the $\sU$-initial structure on $\bbX$ with respect to the $\sU$-structured source $(f_i:\bbX\lra\bbY_i)_{i\in I}$.
\end{proof}

In particular, $\sU$ and $\sU^{\op}$ both being topological implies that $\sU$ has a fully faithful left adjoint $\QCat\lra\QCatCls$ which provides a $\CQ$-category $\bbX$ with the \emph{discrete} $\CQ$-closure space $(\bbX,1_{\PX})$ (i.e., every $\mu\in\PX$ is closed), and a fully faithful right adjoint $\QCat\lra\QCatCls$ which endows $\bbX$ with the \emph{indiscrete} $\CQ$-closure space $(\bbX,\top_{\PX})$, where
$$\top_{\PX}(\mu)(x)=\top_{|x|,|\mu|}$$
for all $\mu\in\PX$ and $x\in\bbX_0$.

A locally small category $\CC$ is \emph{totally cocomplete} \cite{Borger1990} if each diagram $D:\CJ\lra\CC$ (here $\CJ$ is possibly large) has a colimit in $\CC$ whenever the colimit of $\CC(X,D-)$ exists in $\Set$ for all $X\in\ob\CC$. $\CC$ is totally cocomplete if and only if $\CC$ is \emph{total} \cite{Street1978}; that is, the Yoneda embedding $\CC\lra\Set^{\CC^{\op}}$ has a left adjoint. $\CC$ is \emph{totally complete} (or equivalently, \emph{cototal}) if $\CC^{\op}$ is totally cocomplete. Moreover, it is already known in category theory that
\begin{itemize}
\item if $U:\CE\lra\CB$ is a topological functor and $\CB$ is totally cocomplete, then so is $\CE$ (see \cite[Theorem 6.13]{Kelly1986});
\item $\QCat$ is a totally cocomplete and totally complete category (see \cite[Theorem 2.7]{Shen2015a}).
\end{itemize}
Thus we conclude:

\begin{cor} \label{QCatCls_total}
$\QCatCls$ is totally cocomplete and totally complete and, in particular, cocomplete and complete.
\end{cor}

\subsection{$\CQ$-closure spaces with discrete underlying $\CQ$-categories}

By restricting the objects of $\QCatCls$ to those $\CQ$-closure spaces with discrete underlying $\CQ$-categories (i.e., $\CQ_0$-typed sets), we obtain a full subcategory of $\QCatCls$ and denote it by $\QCls$, whose morphisms are continuous type-preserving maps, or \emph{continuous maps} for short\footnote{Our notations here deviate from \cite{Shen2014,Shen2013a}, where $\QCls$ is in fact $\QCatCls$ in this paper, and our $\QCls$ here did not appear in \cite{Shen2014,Shen2013a}.}.

There is a natural functor $(-)_0:\QCatCls\lra\QCls$ that sends a $\CQ$-closure space $(\bbX,c)$ to $(\bbX_0,c_0)$ with
$$c_0:\PX_0\lra\PX_0,\quad \mu\mapsto c(\mu\circ\bbX).$$
To see the functoriality of $(-)_0$, first note that the $\CQ$-closure spaces $(\bbX,c)$, $(\bbX_0,c_0)$ have exactly the same closed presheaves:

\begin{prop} \label{cPX=c0PX0}
For each $\CQ$-closure space $(\bbX,c)$, $\sC(\bbX,c)=\sC(\bbX_0,c_0)$.
\end{prop}

\begin{proof}
It suffices to show that $\mu\in\sC(\bbX_0,c_0)$ implies $\mu\in\PX$. Suppose $\mu:\bbX_0\oto\{|\mu|\}$ satisfies $c_0\mu=\mu$, then
$$\mu\circ\bbX\leq c(\mu\circ\bbX)=c_0\mu=\mu,$$
and thus $\mu\circ\bbX=\mu$ since the reverse inequality is trivial. Therefore, $\mu$ is a $\CQ$-distributor $\bbX\oto\{|\mu|\}$ (see Equation (\ref{distributor_def}) for the definition), i.e., $\mu\in\PX$.
\end{proof}

Consequently, the functoriality of $(-)_0:\QCatCls\lra\QCls$ follows from the above observation and Proposition \ref{continuous_functor_condition}(iv), i.e., the continuity of a $\CQ$-functor $f:(\bbX,c)\lra(\bbY,d)$ implies the continuity of its underlying type-preserving map $f:(\bbX_0,c_0)\lra(\bbY_0,d_0)$.

\begin{rem}
Although $(\bbX,c)$ and $(\bbX_0,c_0)$ have the same closed presheaves, in general they are not isomorphic objects in the category $\QCatCls$ if $\bbX$ is non-discrete.
\end{rem}

\begin{prop} \label{QCatCls_QCls}
$\QCls$ is a coreflective subcategory of $\QCatCls$.
\end{prop}

\begin{proof}
For all $(X,c)\in\ob(\QCls)$ and $(\bbY,d)\in\ob(\QCatCls)$, by Propositions \ref{continuous_functor_condition}(iv) and \ref{cPX=c0PX0} one soon has
$$\QCatCls((X,c),(\bbY,d))\cong\QCls((X,c),(\bbY_0,d_0)).$$
Hence, $(-)_0$ is right adjoint to the inclusion functor $\QCls\ \to/^(->/\QCatCls$.
\end{proof}

Since $\Set\da\CQ_0$ is a totally cocomplete and totally complete category\footnote{The total (co)completeness of $\Set\da T$ for any set $T$ follows from its (co)completeness, (co)wellpoweredness and the existence of a (co)generating set (see \cite[Corollary 3.5]{Borger1990}).}, by replacing every $\CQ$-category in the proof of Proposition \ref{U_topological} with discrete ones and repeating the reasoning for Corollary \ref{QCatCls_total}, one immediately deduces that $\QCls$ is also a well-behaved category as $\QCatCls$:

\begin{prop} \label{QCls_total}
The forgetful functor $\QCls\lra\Set\da\CQ_0$ is topological. Therefore, $\QCls$ is totally cocomplete and totally complete and, in particular, cocomplete and complete.
\end{prop}

\subsection{Specialization $\CQ$-categories and $\CQ$-Alexandrov spaces} \label{Specialization_QCat}

For all $\CQ$-categories $\bbX$, $\bbY$, the adjunction $(-)^{\nat}\dv(-)^*$ in Proposition \ref{star_graph_adjoint} gives rise to an isomorphism
$$\QDist(\bbX,\bbY)\cong\QCat(\bbY,\PX).$$
Explicitly, each $\CQ$-distributor $\phi:\bbX\oto\bbY$ has a \emph{transpose}
\begin{equation} \label{tphi_def}
\tphi:\bbY\lra\PX,\quad\tphi y=\phi(-,y);
\end{equation}
and correspondingly, the transpose of each $\CQ$-functor $f:\bbY\lra\PX$ is given by
\begin{equation} \label{tf_def}
\tf:\bbX\oto\bbY,\quad \tf(x,y)=(fy)(x).
\end{equation}

Now let $X$ be a $\CQ_0$-typed set and $(X,c)$ a $\CQ$-closure space, the $\CQ$-closure operator $c:\sP X\lra\sP X$ has a transpose
$$\tc:X\oto\sP X.$$

\begin{lem} \label{tc_rda_tc}
The $\CQ$-relation $\tc\rda\tc:X\oto X$ may be calculated as
$$(\tc\rda\tc)(x,y)=\bw_{\mu\in\sC(X,c)}\mu(y)\rda\mu(x)$$
for all $x,y\in X$.
\end{lem}

\begin{proof}
This is easy since
$$(\tc\rda\tc)(x,y)=\tc(y,-)\rda\tc(x,-)=\bw_{\mu\in\sP X}(c\mu)(y)\rda(c\mu)(x)=\bw_{\mu\in\sC(X,c)}\mu(y)\rda\mu(x).$$
\end{proof}

The $\CQ$-relation $\tc\rda\tc$ on $X$ defines a $\CQ$-category $(X,\tc\rda\tc)$, which is functorial from $\QCls$ to $\QCat$:

\begin{prop}
The map $(X,c)\mapsto(X,\tc\rda\tc)$ defines a functor $\sS:\QCls\lra\QCat$.
\end{prop}

\begin{proof}
$(X,\tc\rda\tc)$ is obviously a $\CQ$-category. Now let $f:(X,c)\lra(Y,d)$ be a continuous map in $\QCls$, we show that $f:(X,\tc\rda\tc)\lra(Y,\td\rda\td)$ is a $\CQ$-functor. Indeed, for all $x,x'\in X$,
\begin{align*}
(\tc\rda\tc)(x,x')&=\bw_{\mu\in\sC(X,c)}\mu(x')\rda\mu(x)&(\text{Lemma \ref{tc_rda_tc}})\\
&\leq\bw_{\lam\in\sC(Y,d)}(f^{\la}\lam)(x')\rda(f^{\la}\lam)(x)&(\text{Proposition \ref{continuous_functor_condition}(iv)})\\
&=\bw_{\lam\in\sC(Y,d)}\lam(fx')\rda\lam(fx)&(\text{Equation (\ref{fla_lam})})\\
&=(\td\rda\td)(fx,fx'),&(\text{Lemma \ref{tc_rda_tc}})
\end{align*}
as desired.
\end{proof}

We call $(X,\tc\rda\tc)$ the \emph{specialization $\CQ$-category}\footnote{The author is indebted to Professor Dexue Zhang for enlightening suggestions on the definition of specialization $\CQ$-categories.} of a $\CQ$-closure space $(X,c)$. The intuition of this term is from the specialization (pre)order:

\begin{exmp}
For a closure space $(X,c)$ with $X$ a crisp set and $c$ a closure operator on ${\bf 2}^X$, the specialization (pre)order on $X$ is given by
$$x\leq y\iff x\in c\{y\}.$$
Now consider $c$ as a relation $\tc:X\oto{\bf 2}^X$ (i.e., $\tc\subseteq X\times {\bf 2}^X$), the implication $\tc\rda\tc$ in the quantaloid $\Rel$ (as a special case of the implication in $\QRel$) is exactly
$$\tc\rda\tc=\{(x,y)\in X\times X\mid\forall A\in{\bf 2}^X:\ y\in c(A){}\Lra{}x\in c(A)\}.$$
Since it is easy to check $(x,y)\in\tc\rda\tc\iff x\in c\{y\}$, it follows that when $\CQ={\bf 2}$, our definition of specialization ${\bf 2}$-categories coincides with the notion of specialization order on the set of points of a closure space.
\end{exmp}

Conversely, let $\bbX$ be any $\CQ$-category. The $\CQ$-relation
$$\bbX:\bbX_0\oto\bbX_0$$
on $\bbX_0$ generates a $\CQ$-functor $\bbX^*:\PX_0\lra\PX_0$, which gives rise to a $\CQ$-closure space $(\bbX_0,\bbX^*)$. Intuitively, $\bbX^*$ turns any $\mu\in\PX_0$ into a presheaf $\mu\circ\bbX$ on $\bbX$; that is, $\sC(\bbX_0,\bbX^*)=\PX$.

\begin{prop} \label{D_fully_faithful}
Let $\bbX$, $\bbY$ be $\CQ$-categories and $f:\bbX_0\lra\bbY_0$ a type-preserving map. Then $f:\bbX\lra\bbY$ is a $\CQ$-functor if, and only if, $f:(\bbX_0,\bbX^*)\lra(\bbY_0,\bbY^*)$ is a continuous map.
\end{prop}

\begin{proof}
The necessity is easy. For the sufficiency, note that for all $x,x'\in\bbX_0$, $\sy_{\bbY}fx'\in\PY$ implies $f^{\la}\sy_{\bbY}fx'\in\PX$ since $f:(\bbX_0,\bbX^*)\lra(\bbY_0,\bbY^*)$ is continuous (see Proposition \ref{continuous_functor_condition}(iv)); consequently
\begin{align*}
\bbX(x,x')&\leq\bbY(fx',fx')\circ\bbX(x,x')\\
&=(\sy_{\bbY}fx')(fx')\circ\bbX(x,x')\\
&=(f^{\la}\sy_{\bbY}fx')(x')\circ\bbX(x,x')&(\text{Equation (\ref{fla_lam})})\\
&\leq(f^{\la}\sy_{\bbY}fx')(x)&(f^{\la}\sy_{\bbY}fx'\in\PX)\\
&=(\sy_{\bbY}fx')(fx)&(\text{Equation (\ref{fla_lam})})\\
&=\bbY(fx,fx'),
\end{align*}
indicating that $f:\bbX\lra\bbY$ is a $\CQ$-functor.
\end{proof}

Proposition \ref{D_fully_faithful} shows that the assignment $\bbX\mapsto(\bbX_0,\bbX^*)$ defines a fully faithful functor
$$\sD:\QCat\lra\QCls$$
which embeds $\QCat$ in $\QCls$ as a full subcategory. This embedding is coreflective and $\sS$ is the coreflector:

\begin{thm} \label{D_S_adjoint}
$\sD:\QCat\lra\QCls$ is a left adjoint and right inverse of $\sS:\QCls\lra\QCat$.
\end{thm}

\begin{proof}
For each $\CQ$-category $\bbX$, one asserts that $\bbX=(\bbX_0,\widetilde{\bbX^*}\rda\widetilde{\bbX^*})=\sS\sD\bbX$ since
$$\bbX(x,y)=\bw_{\mu\in\PX}\mu(y)\rda\mu(x)=\bw_{\mu\in\sC(\bbX_0,\bbX^*)}\mu(y)\rda\mu(x)=(\widetilde{\bbX^*}\rda\widetilde{\bbX^*})(x,y)$$
for all $x,y\in\bbX_0$, where the last equality follows from Lemma \ref{tc_rda_tc}. Conversely, for a $\CQ$-closure space $(X,c)$, by definition one has $\sD\sS(X,c)=(X,(\tc\rda\tc)^*)$. Note that for all $\mu\in\sP X$,
\begin{align*}
(\tc\rda\tc)^*\mu&\leq(c\mu)(\tc\rda\tc)&(1_{\sP X}\leq c)\\
&=\tc(-,\mu)\circ(\tc\rda\tc)&(\text{Equation (\ref{tf_def})})\\
&=(\tc\circ(\tc\rda\tc))(-,\mu)\\
&\leq\tc(-,\mu)\\
&=c\mu,&(\text{Equation (\ref{tf_def})})
\end{align*}
thus $1_X:(X,(\tc\rda\tc)^*)\lra(X,c)$ is a continuous $\CQ$-functor. Finally, it is easy to check that $\{1_{\bbX}:\bbX\lra\sS\sD\bbX\}_{\bbX\in\ob(\QCat)}$ and $\{1_X:\sD\sS(X,c)\lra(X,c)\}_{(X,c)\in\ob(\QCls)}$ are both natural transformations and satisfy the triangle identities (see \cite[Theorem IV.1.2]{MacLane1998}), thus they are respectively the unit and counit of the adjunction $\sD\dv\sS$.
\end{proof}

A $\CQ$-closure space $(X,c)$ is called a \emph{$\CQ$-Alexandrov space} if the inclusion $j:\sC(X,c)\ \to/^(->/\sP X$ not only has a left adjoint (i.e., the codomain restriction $\oc:\sP X\lra\sC(X,c)$ of $c$), but also has a right adjoint in $\QCat$; or equivalently, if $\sC(X,c)$ is both a $\CQ$-closure system and a $\CQ$-interior system of $\sP X$. The following proposition follows immediately from Proposition \ref{closure_system_condition} and its dual, together with Example \ref{PX_tensor_complete}:

\begin{prop} \label{Alexandrov_closure_system}
Let $X$ be a $\CQ_0$-typed set and $\bbA$ a $\CQ$-subcategory of $\sP X$. Then $\bbA$ defines a $\CQ$-Alexandrov space on $X$ if, and only if,
\begin{itemize}
\item[\rm (1)] $u\circ\mu\in\bbA_0$ for all $\mu\in\bbA_0$ and $u\in\sP\{|\mu|\}$,
\item[\rm (2)] $v\rda\mu\in\bbA_0$ for all $\mu\in\bbA_0$ and $v\in\sPd\{|\mu|\}$,
\item[\rm (3)] $\displaystyle\bv_{i\in I}\mu_i\in\bbA_0$ and $\displaystyle\bw_{i\in I}\mu_i\in\bbA_0$ for all $\{\mu_i\}_{i\in I}\subseteq\bbA_q$ $(q\in\CQ_0)$.
\end{itemize}
\end{prop}

\begin{exmp} \label{Alexandrov_exmp}
For any $\CQ$-category $\bbX$, $\sD\bbX=(\bbX_0,\bbX^*)$ is a $\CQ$-Alexandrov space.
\end{exmp}

\begin{prop} \label{Alexandrov_DS}
A $\CQ$-closure space $(X,c)$ is a $\CQ$-Alexandrov space if, and only if, $(X,c)=\sD\sS(X,c)$.
\end{prop}

\begin{proof}
The sufficiency follows immediately from Example \ref{Alexandrov_exmp}. For the necessity, since it is already known in the proof of Theorem \ref{D_S_adjoint} that $(\tc\rda\tc)^*\leq c$, we only need to prove $c\leq(\tc\rda\tc)^*$. Indeed, for all $\mu\in\sP X$, Proposition \ref{Alexandrov_closure_system} guarantees that
$$(\tc\rda\tc)^*\mu=\mu\circ(\tc\rda\tc)=\bv_{x\in X}\mu(x)\circ\Big(\bw_{\lam\in\sP X}(c\lam)(x)\rda(c\lam)\Big)\in\sC(X,c),$$
and consequently $c\mu\leq c(\tc\rda\tc)^*\mu=(\tc\rda\tc)^*\mu$.
\end{proof}



\begin{rem} \label{distributor_specialization}
For a general $\CQ$-closure space $(\bbX,c)$, its specialization $\CQ$-category may be defined as
\begin{equation} \label{tc_tc0}
(\bbX_0,\widetilde{c_0}\rda\widetilde{c_0})=(\bbX_0,\tc\rda\tc)
\end{equation}
since Proposition \ref{cPX=c0PX0} and Lemma \ref{tc_rda_tc} imply
$$(\widetilde{c_0}\rda\widetilde{c_0})(x,y)=\bw_{\mu\in\sC(\bbX_0,c_0)}\mu(y)\rda\mu(x)=\bw_{\mu\in\sC(\bbX,c)}\mu(y)\rda\mu(x)=(\tc\rda\tc)(x,y)$$
for all $x,y\in\bbX_0$. Here $\tc\rda\tc$ is a coarser $\CQ$-category on $\bbX_0$ in the sense that $\bbX\leq\tc\rda\tc$ always holds.

Note that the above definition only relies on the $\CQ$-distributor $\tc:\bbX\oto\PX$. In fact, one may define a specialization $\CQ$-category
$$(\bbX_0,\phi\rda\phi)$$
for any given $\CQ$-distributor $\phi:\bbX\oto\bbY$, which is obviously a coarser $\CQ$-category than $\bbX$ on the $\CQ_0$-typed set $\bbX_0$. In this way the realm of the specialization order would be largely extended and deservers further investigation in the future.
\end{rem}

\subsection{$\CQ$-closure spaces and complete $\CQ$-categories}

In this subsection, we incorporate and enhance some results in \cite{Shen2013a} to demonstrate the relation between the categories $\QCatCls$ and $\QSup$.

First, we establish the 2-functoriality of the assignment
$$(\bbX,c)\mapsto\sC(\bbX,c).$$
For a continuous $\CQ$-functor $f:(\bbX,c)\lra(\bbY,d)$, Proposition \ref{continuous_functor_condition}(iv) shows that the $\CQ$-functor $f^{\la}:\PY\lra\PX$ may be restricted to $$\ftla:\sC(\bbY,d)\lra\sC(\bbX,c),\quad\lam\mapsto f^{\la}\lam,$$
and it is in fact right adjoint to the composite $\CQ$-functor (recall that $\od$ is the codomain restriction of the $\CQ$-closure operator $d:\PY\lra\PY$)
$$\ftra:=(\sC(\bbX,c)\ \to/^(->/\PX\to^{f^{\ra}}\PY\to^{\od}\sC(\bbY,d))$$
as the following proposition reveals, which will be proved as a special case of Proposition \ref{czeta_la} in the next section:

\begin{prop} {\rm\cite{Shen2013a}} \label{dfra_fla_adjoint}
For a continuous $\CQ$-functor $f:(\bbX,c)\lra(\bbY,d)$,
$$\ftra\dv\ftla:\sC(\bbY,d)\lra\sC(\bbX,c).$$
\end{prop}

Consequently, one may easily deduce that the assignments $(\bbX,c)\mapsto\sC(\bbX,c)$ and $f\mapsto\ftra$ induce a 2-functor
$$\sC:\QCatCls\lra\QSup.$$

Conversely, for a complete $\CQ$-category $\bbX$, since $\sup_{\bbX}\dv\sy_{\bbX}$,
$$c_{\bbX}:=\sy_{\bbX}{\sup}_{\bbX}:\PX\lra\PX$$
is a $\CQ$-closure operator on $\PX$ and thus one has a $\CQ$-closure space $(\bbX,c_{\bbX})$.

\begin{prop} \label{la_continuous}
For complete $\CQ$-categories $\bbX$, $\bbY$, a $\CQ$-functor $f:\bbX\lra\bbY$ is $\sup$-preserving if, and only if, $f:(\bbX,c_{\bbX})\lra(\bbY,c_{\bbY})$ is a continuous $\CQ$-functor.
\end{prop}

\begin{proof}
First note that if $\bbX$, $\bbY$ are both complete, then
$$f^{\ra}\leq f^{\ra}\sy_{\bbX}{\sup}_{\bbX}=\sy_{\bbY}f{\sup}_{\bbX},$$
where the first inequality holds since $\sup_{\bbX}\dv\sy_{\bbX}$, and the second equality follows from Proposition \ref{Yoneda_natural}. Together with $\sup_{\bbY}\dv\sy_{\bbY}$ one concludes
\begin{equation} \label{supYfra_leq_fsupX}
{\sup}_{\bbY}f^{\ra}\leq f{\sup}_{\bbX}.
\end{equation}
Thus, $f:\bbX\lra\bbY$ is $\sup$-preserving if and only if $f\sup_{\bbX}\leq\sup_{\bbY}f^{\ra}$ since the reverse inequality (\ref{supYfra_leq_fsupX}) always holds. Note further that
\begin{align*}
f{\sup}_{\bbX}\leq{\sup}_{\bbY}f^{\ra}&\iff\sy_{\bbY}f{\sup}_{\bbX}\leq\sy_{\bbY}{\sup}_{\bbY}f^{\ra}\\
&\iff f^{\ra}\sy_{\bbX}{\sup}_{\bbX}\leq\sy_{\bbY}{\sup}_{\bbY}f^{\ra}&(\text{Proposition \ref{Yoneda_natural}})\\
&\iff f^{\ra}c_{\bbX}\leq c_{\bbY}f^{\ra},
\end{align*}
and hence, the conclusion follows.
\end{proof}

Proposition \ref{la_continuous} gives rise to a fully faithful 2-functor
$$\sI:\QSup\lra\QCatCls,\quad\bbX\mapsto(\bbX,c_{\bbX})$$
that embeds $\QSup$ in $\QCatCls$ as a full 2-subcategory. In fact, this embedding is reflective with $\sC$ the reflector:

\begin{thm} {\rm\cite{Shen2013a}} \label{C_I_adjoint}
$\sC:\QCatCls\lra\QSup$ is a left inverse (up to isomorphism) and left adjoint of $\sI:\QSup\lra\QCatCls$.
\end{thm}

Although a proof of this theorem can be found in \cite{Shen2013a}, here we provide an easier alternative proof:

\begin{proof}[Proof of Theorem \ref{C_I_adjoint}]
Note that for any separated complete $\CQ$-category $\bbX$,
$$\sC\sI\bbX=\sC(\bbX,c_{\bbX})=\{\sy_{\bbX}x\mid x\in\bbX_0\}.$$
Thus $\sup_{\bbX}:\sC(\bbX,c_{\bbX})\lra\bbX$ is clearly an isomorphism (and in particular a left adjoint) in $\QCat$, with the codomain restriction $\sy_{\bbX}:\bbX\lra \sC(\bbX,c_{\bbX})$ of the Yoneda embedding as its inverse. Moreover, $\{\sup_{\bbX}:\sC\sI\bbX\lra\bbX\}_{\bbX\in\ob(\QSup)}$ is a 2-natural transformation as one easily derives from Proposition \ref{la_sup_preserving}. Therefore, $\sC\sI$ is naturally isomorphic to the identity 2-functor on $\QSup$, and it remains to show that $\{\sup_{\bbX}\}_{\bbX\in\ob(\QSup)}$ is the counit of the adjunction $\sC\dv\sI$.

To this end, taking any $\CQ$-closure space $(\bbY,d)$ and left adjoint $\CQ$-functor $f:\sC(\bbY,d)\lra\bbX$, one must find a unique continuous $\CQ$-functor $g:(\bbY,d)\lra(\bbX,c_\bbX)$ that makes the following diagram commute:
\begin{equation} \label{supX_Cg_f}
\bfig
\btriangle/-->`->`->/<1200,500>[\sC(\bbY,d)`\sC\sI\bbX=\sC(\bbX,c_{\bbX})`\bbX;\sC g=\gtra`f`\sup_{\bbX}]
\efig
\end{equation}

For this, one defines $g$ as the composite
$$g:=(\bbY\to^{\sy_{\bbY}}\PY\to^{\od}\sC(\bbY,d)\to^f\bbX).$$
Note that $f\od:\PY\lra\bbX$ is a left adjoint $\CQ$-functor since both $f$ and $\od$ are left adjoints (for $\od$, see Proposition \ref{closure_operator_adjoint_inclusion}). Thus
\begin{align*}
f\lam&=f\od\lam&(\lam\in\sC(\bbY,d))\\
&=f\od\sy_{\bbY}^{\la}\sy_{\bbY}^{\ra}\lam&(\text{Proposition \ref{fully_faithful_graph}(iii)})\\
&=f\od\,{\sup}_{\PY}\sy_{\bbY}^{\ra}\lam&(\text{Example \ref{PX_tensor_complete}})\\
&={\sup}_{\bbX}(f\od)^{\ra}\sy_{\bbY}^{\ra}\lam&(\text{Proposition \ref{la_sup_preserving}})\\
&={\sup}_{\bbX}c_{\bbX}g^{\ra}\lam&({\sup}_{\bbX}\dv\sy_{\bbX})\\
&={\sup}_{\bbX}\gtra\lam
\end{align*}
for all $\lam\in\sC(\bbY,d)$ and
\begin{align*}
g^{\ra}d&\leq\sy_{\bbX}{\sup}_{\bbX}c_{\bbX}g^{\ra}d&({\sup}_{\bbX}\dv\sy_{\bbX})\\
&=\sy_{\bbX}{\sup}_{\bbX}\gtra\od\\
&=\sy_{\bbX}f\od\\
&=\sy_{\bbX}f\od\sy_{\bbY}^{\la}\sy_{\bbY}^{\ra}&(\text{Proposition \ref{fully_faithful_graph}(iii)})\\
&=\sy_{\bbX}f\od\,{\sup}_{\PY}\sy_{\bbY}^{\ra}&(\text{Example \ref{PX_tensor_complete}})\\
&=\sy_{\bbX}{\sup}_{\bbX}(f\od)^{\ra}\sy_{\bbY}^{\ra}&(\text{Proposition \ref{la_sup_preserving}})\\
&=c_{\bbX}g^{\ra}.
\end{align*}
Hence, $g:(\bbY,d)\lra(\bbX,c_\bbX)$ is continuous and the diagram (\ref{supX_Cg_f}) commutes.

For the uniqueness of $g$, suppose there is another continuous $\CQ$-functor $h:(\bbY,d)\lra(\bbX,c_\bbX)$ that makes the diagram (\ref{supX_Cg_f}) commute. Then
\begin{align*}
h&={\sup}_{\bbX}\sy_{\bbX}h&({\sup}_{\bbX}\sy_{\bbX}=1_{\bbX})\\
&={\sup}_{\bbX}h^{\ra}\sy_{\bbY}&(\text{Proposition \ref{Yoneda_natural}})\\
&={\sup}_{\bbX}c_{\bbX}h^{\ra}\sy_{\bbY}&({\sup}_{\bbX}\dv\sy_{\bbX})\\
&={\sup}_{\bbX}c_{\bbX}h^{\ra}d\sy_{\bbY}&(\text{Proposition \ref{continuous_functor_condition}(ii)})\\
&={\sup}_{\bbX}h^{\triangleright}\od\sy_{\bbY}\\
&=f\od\sy_{\bbY}&(\text{commutativity of the diagram (\ref{supX_Cg_f})})\\
&=g,
\end{align*}
where ${\sup}_{\bbX}\sy_{\bbX}=1_{\bbX}$ may be easily verified since $\bbX$ is separated, completing the proof.
\end{proof}

Let $\sC_0:\QCls\lra\QSup$ denote the restriction of $\sC$ on $\QCls$, and $\sI_0:\QSup\lra\QCls$ the composition of $(-)_0:\QCatCls\lra\QCls$ and $\sI$, we have:

\begin{cor} \label{C0_I0_adjoint}
$\sC_0:\QCls\lra\QSup$ is a left inverse (up to isomorphism) and left adjoint of $\sI_0:\QSup\lra\QCls$.
\end{cor}

\begin{proof}
$\sC_0\dv\sI_0$ follows from Proposition \ref{QCatCls_QCls} and Theorem \ref{C_I_adjoint}. Furthermore, Proposition \ref{cPX=c0PX0} ensures that
$$\sC_0\sI_0\bbX=\sC(\bbX_0,(c_{\bbX})_0)=\sC(\bbX,c_{\bbX})=\sC\sI\bbX\cong\bbX$$
for all separated complete $\CQ$-categories $\bbX$.
\end{proof}

Finally, Theorem \ref{D_S_adjoint} and Corollary \ref{C0_I0_adjoint} imply that the functor
$$\sP:=(\QCat\to^{\sD}\QCls\to^{\sC_0}\QSup)$$
is left adjoint to the composite functor $\QSup\to^{\sI_0}\QCls\to^{\sS}\QCat$. The following proposition shows that $\sP$ can be obtained by restricting the codomain of $(-)^{\ra}:\QCat\lra\QCat$ to $\QSup$, and $\sS\sI_0$ is the inclusion functor $\QSup\ \to/^(->/\QCat$:
$$\bfig
\qtriangle|alr|/@{->}@<2pt>`@{->}@<-2pt>`@{->}@<2pt>/<1000,500>[\QCat`\QCls`\QSup;\sD`\sP`\sC_0]
\qtriangle|brl|/@{<-}@<-2pt>`@{<-^)}@<2pt>`@{<-}@<-2pt>/<1000,500>[\QCat`\QCls`\QSup;\sS``\sI_0]
\efig$$

\begin{prop} \label{P=C0D}
\begin{itemize}
\item[\rm (1)] For any $\CQ$-functor $f:\bbX\lra\bbY$, $\sP(f:\bbX\lra\bbY)=(f^{\ra}:\PX\lra\PY)$.
\item[\rm (2)] For any left adjoint $\CQ$-functor $f:\bbX\lra\bbY$ between complete $\CQ$-categories, $\sS\sI_0(f:\bbX\lra\bbY)=(f:\bbX\lra\bbY)$.
\end{itemize}
\end{prop}

\begin{proof}
(1) is easy. For (2), it suffices to show that
$$\sS\sI_0\bbX=(\bbX_0,\widetilde{(c_{\bbX})_0}\rda\widetilde{(c_{\bbX})_0})=(\bbX_0,\widetilde{c_{\bbX}}\rda\widetilde{c_{\bbX}})=\bbX,$$
where the second equality follows from Equation (\ref{tc_tc0}). Indeed, for all $x,x'\in\bbX_0$,
\begin{align*}
(\widetilde{c_{\bbX}}\rda\widetilde{c_{\bbX}})(x,x')&=\bw_{\mu\in\sC(\bbX,c_{\bbX})}\mu(x')\rda\mu(x)&(\text{Lemma \ref{tc_rda_tc}})\\
&=\bw_{x''\in\bbX_0}(\sy_{\bbX}x'')(x')\rda(\sy_{\bbX}x'')(x)&(\sC(\bbX,c_{\bbX})=\{\sy_{\bbX}x\mid x\in\bbX_0\})\\
&=\bw_{x''\in\bbX_0}\bbX(x',x'')\rda\bbX(x,x'')\\
&=\bbX(x,x'),
\end{align*}
the conclusion thus follows.
\end{proof}

\begin{rem}
$\sP:\QCat\lra\QSup$ is known as the \emph{free cocompletion functor} \cite{Stubbe2005} of $\CQ$-categories\footnote{From the viewpoint of category theory, a $\CQ$-category $\bbX$ should be called \emph{cocomplete} if every $\mu\in\PX$ admits a supremum, and it is \emph{complete} if every $\lam\in\PdX$ has an infimum. But since a $\CQ$-category is cocomplete if and only if it is complete as we point out in Subsection \ref{Presheaves}, we do not distinguish cocompleteness and completeness of $\CQ$-categories in this paper.}, where $\sP$ is ``free'' since it is left adjoint to the forgetful functor $\QSup\lra\QCat$ (i.e., the inclusion functor). Proposition \ref{P=C0D} in fact provides a factorization of $\sP$ through $\QCls$.
\end{rem}

\section{Continuous $\CQ$-distributors} \label{Continuous_distributor}

In this section, we generalize continuous $\CQ$-functors to continuous $\CQ$-distributors as morphisms of $\CQ$-closure spaces.

\subsection{From continuous $\CQ$-functors to continuous $\CQ$-distributors}

Since $f^{\ra}=(f^{\nat})^*$ (see the definition of $f^{\ra}$ in Subsection \ref{Presheaves}), the continuity of a $\CQ$-functor $f:(\bbY,d)\lra(\bbX,c)$ between $\CQ$-closure spaces is completely characterized by the cograph $f^{\nat}:\bbX\oto\bbY$ of $f$, i.e.,
$$(f^{\nat})^*d\leq c(f^{\nat})^*:\PY\lra\PX.$$
If $f^{\nat}$ is replaced by an arbitrary $\CQ$-distributor $\zeta:\bbX\oto\bbY$, we have the following definition:

\begin{defn} \label{ClsDist_def}
A \emph{continuous $\CQ$-distributor} $\zeta:(\bbX,c)\oto(\bbY,d)$ between $\CQ$-closure spaces is a $\CQ$-distributor $\zeta:\bbX\oto\bbY$ such that $$\zeta^*d\leq c\zeta^*:\PY\lra\PX.$$
\end{defn}

With the local order inherited from $\QDist$, $\CQ$-closure spaces and continuous $\CQ$-distributors constitute a (large) quantaloid $\QClsDist$, for it is easy to verify that compositions and joins of continuous $\CQ$-distributors are still continuous $\CQ$-distributors.

Since the topologicity of a faithful functor $U:\CE\lra\CB$ is equivalent to the topologicity of $U^{\op}:\CE^{\op}\lra\CB^{\op}$ (see above Proposition \ref{U_topological}), $U$ is topological if all \emph{$U$-structured sinks} admit \emph{$U$-final liftings}, where $U$-structured sinks and $U$-final liftings are respectively given by $U^{\op}$-structured sources and $U^{\op}$-initial liftings. Explicitly, $U$ is topological if every $U$-structured sink $(f_i:UX_i\lra S)_{i\in I}$ admits a $U$-final lifting $(\check{f_i}:X_i\lra Y)_{i\in I}$ in the sense that any $\CB$-morphism $g:S\lra UZ$ lifts to an $\CE$-morphism $\check{g}:Y\lra Z$ as soon as every $\CB$-morphism $gf_i:UX_i\lra UZ$ lifts to an $\CE$-morphism $h_i:X_i\lra Z$ ($i\in I$).
$$\bfig
\qtriangle<800,500>[X_i`Y`Z;\check{f_i}`h_i`\check{g}]
\qtriangle(2000,0)<800,500>[UX_i`S`UZ;f_i`Uh_i`g]
\morphism(1300,250)/|->/<400,0>[`;U]
\efig$$
In this way we are able to prove:

\begin{prop} \label{Ud_topological}
The forgetful functor $\sU:\QClsDist\lra\QDist$ is topological.
\end{prop}

\begin{proof}
$\sU$ is obviously faithful. Given a (possibly large) family of $\CQ$-closure spaces $(\bbX_i,c_i)$ and $\CQ$-distributors $\zeta_i:\bbX_i\oto\bbY$ $(i\in I)$, we must find a $\CQ$-closure space $(\bbY,d)$ such that
\begin{itemize}
\item every $\zeta_i:(\bbX_i,c_i)\oto(\bbY,d)$ is a continuous $\CQ$-distributor, and
\item for every $\CQ$-closure space $(\bbZ,e)$, any $\CQ$-distributor $\eta:\bbY\oto\bbZ$ becomes a continuous $\CQ$-distributor $\eta:(\bbY,d)\oto(\bbZ,e)$ whenever all $\eta\circ\zeta_i:(\bbX_i,c_i)\oto(\bbZ,e)$ $(i\in I)$ are continuous $\CQ$-distributors.
\end{itemize}
To this end, one simply defines $d=\displaystyle\bw_{i\in I}(\zeta_i)_*c_i\zeta_i^*$, i.e., the meet of the composite $\CQ$-distributors
$$\bfig
\morphism[\PY`\PX_i;\zeta_i^*]
\morphism(500,0)[\PX_i`\PX_i;c_i]
\morphism(1000,0)[\PX_i`\PY.;(\zeta_i)_*]
\place(250,0)[\circ] \place(750,0)[\circ] \place(1250,0)[\circ]
\efig$$
Then $d$ is the $\sU$-final structure on $\bbY$ with respect to the $\sU$-structured sink $(\zeta_i:\bbX_i\oto\bbY)_{i\in I}$.
\end{proof}

From the motivation of continuous $\CQ$-distributors one has an obvious contravariant 2-functor
\begin{equation} \label{cograph_Cls}
(-)^{\nat}:\QCatCls\lra(\QClsDist)^{\op}
\end{equation}
that sends a continuous $\CQ$-functor $f:(\bbY,d)\lra(\bbX,c)$ to the continuous $\CQ$-distributor $f^{\nat}:(\bbX,c)\oto(\bbY,d)$. Conversely, the following proposition shows that continuous $\CQ$-distributors can be characterized by continuous $\CQ$-functors, which induces a 2-functor
\begin{equation} \label{Kan_Cls}
(-)^*:(\QClsDist)^{\op}\lra\QCatCls.
\end{equation}

\begin{prop}
Let $(\bbX,c)$, $(\bbY,d)$ be $\CQ$-closure spaces and $\zeta:\bbX\oto\bbY$ a $\CQ$-distributor. Then $\zeta:(\bbX,c)\oto(\bbY,d)$ is a continuous $\CQ$-distributor if, and only if, $\zeta^*:(\PY,d^{\ra})\lra(\PX,c^{\ra})$ is a continuous $\CQ$-functor.
\end{prop}

\begin{proof}
The 2-functoriality of $(-)^{\ra}:\QCat\lra\QCat$ ensures that $(\PX,c^{\ra})$, $(\PY,d^{\ra})$ are both $\CQ$-closure spaces and $\zeta^*d\leq c\zeta^*$ implies $(\zeta^*)^{\ra}d^{\ra}\leq c^{\ra}(\zeta^*)^{\ra}$. To show that $(\zeta^*)^{\ra}d^{\ra}\leq c^{\ra}(\zeta^*)^{\ra}$ implies $\zeta^*d\leq c\zeta^*$, it is not difficult to see
\begin{align*}
\zeta^*d&={\sup}_{\PX}\sy_{\PX}\zeta^*d&({\sup}_{\PX}\sy_{\PX}=1_{\PX})\\
&={\sup}_{\PX}(\zeta^*d)^{\ra}\sy_{\PY}&(\text{Proposition \ref{Yoneda_natural}})\\
&\leq{\sup}_{\PX}(c\zeta^*)^{\ra}\sy_{\PY}\\
&={\sup}_{\PX}\sy_{\PX}c\zeta^*&(\text{Proposition \ref{Yoneda_natural}})\\
&=c\zeta^*,&({\sup}_{\PX}\sy_{\PX}=1_{\PX})
\end{align*}
as desired.
\end{proof}

Here the functors (\ref{cograph_Cls}) and (\ref{Kan_Cls}) may be thought of as being lifted from the functors $(-)^{\nat}:\QCat\lra(\QDist)^{\op}$ and $(-)^*:(\QDist)^{\op}\lra\QCat$ through the topological forgetful functors, as the following commutative diagram illustrates:
$$\bfig
\square|alra|/@{->}@<3pt>`->`->`@{->}@<3pt>/<1200,500>[\QCatCls`(\QClsDist)^{\op}`\QCat`(\QDist)^{\op};(-)^{\nat}`\sU`\sU^{\op}`(-)^{\nat}]
\morphism(1200,500)|b|/@{->}@<3pt>/<-1200,0>[(\QClsDist)^{\op}`\QCatCls;(-)^*]
\morphism(1200,0)|b|/@{->}@<3pt>/<-1200,0>[(\QDist)^{\op}`\QCat;(-)^*]
\efig$$


%

\subsection{Continuous $\CQ$-distributors subsume $\sup$-preserving $\CQ$-functors}

In general, the 2-functor $\sC:\QCatCls\lra\QSup$ (see Proposition \ref{dfra_fla_adjoint}) is not full; that is, not every $\sup$-preserving $\CQ$-functor $\sC(\bbX,c)\lra\sC(\bbY,d)$ is the image of some continuous $\CQ$-functor $f:(\bbX,c)\lra(\bbY,d)$ under $\sC$. However, if we extend $\sC$ along $(-)^{\nat}:\QCatCls\lra(\QClsDist)^{\op}$, then we are able to get a full 2-functor $\Cd:(\QClsDist)^{\op}\lra\QSup$:
\begin{equation} \label{C_cograph_Cd}
\bfig
\btriangle/<-`->`->/<1200,500>[(\QClsDist)^{\op}`\QCatCls`\QSup;(-)^{\nat}`\Cd`\sC]
\efig
\end{equation}

Before proceeding, we present the following characterizations of continuous $\CQ$-distributors:

\begin{prop} \label{continuous_dist_condition}
Let $(\bbX,c)$, $(\bbY,d)$ be $\CQ$-closure spaces and $\zeta:\bbX\oto\bbY$ a $\CQ$-distributor. The following statements are equivalent:
\begin{itemize}
\item[\rm (i)] $\zeta:(\bbX,c)\oto(\bbY,d)$ is a continuous $\CQ$-distributor.
\item[\rm (ii)] $c\zeta^*d\leq c\zeta^*$, thus $c\zeta^*d=c\zeta^*$.
\item[\rm (iii)] $d\zeta_*c\leq\zeta_*c$, thus $d\zeta_*c=\zeta_*c$.
\item[\rm (iv)] $\zeta_*\mu\in\sC(\bbY,d)$ whenever $\mu\in\sC(\bbX,c)$.
\end{itemize}
\end{prop}

\begin{proof}
(i)${}\Lra{}$(ii): If $\zeta^*d\leq c\zeta^*$, then $c\zeta^*d\leq cc\zeta^*=c\zeta^*$.

(ii)${}\Lra{}$(iii): This follows from $\zeta^*d\zeta_*c\leq c\zeta^*d\zeta_*c=c\zeta^*\zeta_*c\leq cc=c$ and $\zeta^*\dv\zeta_*$.

(iii)${}\Lra{}$(i): $\zeta^*d\leq c\zeta^*$ follows immediately from $d\leq d\zeta_* \zeta^*\leq d\zeta_*c\zeta^*=\zeta_*c\zeta^*$ and $\zeta^*\dv\zeta_*$.

(iii)$\iff$(iv) is trivial.
\end{proof}

By Proposition \ref{continuous_dist_condition}(iv), the $\CQ$-functor $\zeta_*:\PX\lra\PY$ may be restricted to a $\CQ$-functor
$$\zetla:\sC(\bbX,c)\lra\sC(\bbY,d),\quad\mu\mapsto\zeta_*\mu.$$
As the general version of Proposition \ref{dfra_fla_adjoint}, we show that
$$\zetra:=(\sC(\bbY,d)\ \to/^(->/\PY\to^{\zeta^*}\PX\to^{\oc}\sC(\bbX,c))$$
is left adjoint to $\zetla$ for any continuous $\CQ$-distributor $\zeta:(\bbX,c)\oto(\bbY,d)$:

\begin{prop} \label{czeta_la}
For a continuous $\CQ$-distributor $\zeta:(\bbX,c)\oto(\bbY,d)$,
$$\zetra\dv\zetla:\sC(\bbX,c)\lra\sC(\bbY,d).$$
\end{prop}

\begin{proof}
It suffices to prove
$$\PX(c\zeta^*\lam,\mu)=\PX(\zeta^*\lam,\mu)$$
for all $\lam\in\sC(\bbY,d)$, $\mu\in\sC(\bbX,c)$ since one already has $\PY(\lam,\zeta_*\mu)=\PX(\zeta^*\lam,\mu)$. Indeed,
\begin{align*}
\PX(\zeta^*\lam,\mu)&\leq\PX(c\zeta^*\lam,c\mu)&(c\ \text{is a}\ \CQ\text{-functor})\\
&=\PX(c\zeta^*\lam,\mu)&(\mu\in\sC(\bbX,c))\\
&=\mu\lda c\zeta^*\lam\\
&\leq\mu\lda \zeta^*\lam&(c\ \text{is a}\ \CQ\text{-closure operator})\\
&=\PX(\zeta^*\lam,\mu),
\end{align*}
and the conclusion follows.
\end{proof}

Now we are ready to show that the assignment $(\bbX,c)\mapsto\sC(\bbX,c)$ induces a full 2-functor
$$\Cd:(\QClsDist)^{\op}\lra\QSup$$
that maps a continuous $\CQ$-distributor $\zeta:(\bbX,c)\oto(\bbY,d)$ to the left adjoint $\CQ$-functor
$$\zetra:\sC(\bbY,d)\lra\sC(\bbX,c),$$
which obviously makes the diagram (\ref{C_cograph_Cd}) commute:

\begin{prop} \label{Cd_functor}
$\Cd:(\QClsDist)^{\op}\lra\QSup$ is a full 2-functor. Moreover, $\Cd$ is a quantaloid homomorphism.
\end{prop}

\begin{proof}
{\bf Step 1.} $\Cd:(\QClsDist)^{\op}\lra\QSup$ is a functor. For this one must check that
$$\zetra\etra=(\Cd\zeta)(\Cd\eta)=\Cd(\eta\circ\zeta)=(\eta\circ\zeta)^{\triangleright}:\sC(\bbZ,e)\lra\sC(\bbX,c)$$
for any continuous $\CQ$-distributors $\zeta:(\bbX,c)\oto(\bbY,d)$, $\eta:(\bbY,d)\oto(\bbZ,e)$. It suffices to show that
$$c\zeta^*d\eta^*=c(\eta\circ\zeta)^*=c\zeta^*\eta^*:\PZ\lra\PX.$$
On one hand, by Definition \ref{ClsDist_def} one immediately has
$$c\zeta^*d\eta^*\leq cc\zeta^*\eta^*=c\zeta^*\eta^*$$
since $c$ is idempotent. On the other hand, $c\zeta^*\eta^*\leq c\zeta^*d\eta^*$ is trivial since $1_{\PY}\leq d$.

{\bf Step 2.} $\Cd:(\QClsDist)^{\op}\lra\QSup$ is full. For all $\CQ$-closure spaces $(\bbX,c)$, $(\bbY,d)$, one needs to show that the map
$$\Cd:\QClsDist((\bbX,c),(\bbY,d))\lra\QSup(\sC(\bbY,d),\sC(\bbX,c))$$
is surjective.

For each left adjoint $\CQ$-functor $f:\sC(\bbY,d)\lra\sC(\bbX,c)$, define a $\CQ$-distributor $\zeta:\bbX\oto\bbY$ through its transpose (see Equation (\ref{tphi_def}))
\begin{equation} \label{tzeta_def}
\tzeta:=(\bbY\to^{\sy_{\bbY}}\PY\to^{\od}\sC(\bbY,d)\to^f\sC(\bbX,c)\ \to/^(->/\PX).
\end{equation}
We claim that $\zeta:(\bbX,c)\oto(\bbY,d)$ is a continuous $\CQ$-distributor and $\Cd\zeta=f$.

First, we show that the diagram
\begin{equation} \label{czeta_fd}
\bfig
\square<800,500>[\PY`\PX`\sC(\bbY,d)`\sC(\bbX,c);\zeta^*`\od`\oc`f]
\efig
\end{equation}
is commutative. Indeed, it follows from Example \ref{PX_tensor_complete} and Corollary \ref{closure_system_sup} that tensors in $\sC(\bbX,c)$ are given by
\begin{equation} \label{tensor_cPX}
u\otimes\mu=c(u\circ\mu)
\end{equation}
for all $\mu\in\sC(\bbX,c)$, $u\in\sP\{|\mu|\}$. In addition, $\oc:\PX\lra\sC(\bbX,c)$ and $\od:\PY\lra\sC(\bbY,d)$ are both left adjoint $\CQ$-functors by Proposition \ref{closure_operator_adjoint_inclusion}, thus so is
$$f\od:\PY\lra\sC(\bbY,d)\lra\sC(\bbX,c).$$
For all $\lam\in\PY$, since the presheaf $\lam\circ\zeta$ can be written as the pointwise join of the $\CQ$-distributors $\lam(y)\circ(\tzeta y)$ $(y\in\bbY_0)$, one has
\begin{align*}
\oc\zeta^*\lam&=\oc(\lam\circ\zeta)\\
&=\oc\Big(\bv_{y\in\bbY_0}\lam(y)\circ(\tzeta y)\Big)\\
&=\oc\Big(\bv_{y\in\bbY_0}\lam(y)\circ(f\od\sy_{\bbY}y)\Big)&\text{(Equation (\ref{tzeta_def}))}\\
&=\bigsqcup_{y\in\bbY_0}\oc(\lam(y)\circ(f\od\sy_{\bbY}y))&\text{(Proposition \ref{la_tensor_preserving})}\\
&=\bigsqcup_{y\in\bbY_0}\lam(y)\otimes(f\od\sy_{\bbY}y)&\text{(Equation (\ref{tensor_cPX}))}\\
&=f\od\Big(\bv_{y\in\bbY_0}\lam(y)\circ(\sy_{\bbY}y)\Big)&\text{(Proposition \ref{la_tensor_preserving})}\\
&=f\od(\lam\circ\bbY)\\
&=f\od\lam,
\end{align*}
where $\bv$ and $\bigsqcup$ respectively denote the underlying joins in $\PX$ and $\sC(\bbX,c)$.

Second, by applying the commutative diagram (\ref{czeta_fd}) one obtains
$$\oc\zeta^*d=f\od d=f\od=\oc\zeta^*\quad\text{and}\quad(\Cd\zeta)\lam=\zetra\lam=\oc\zeta^*\lam=f\od\lam=f\lam$$
for all $\lam\in\sC(\bbY,d)$, where the first equation implies the continuity of $\zeta:(\bbX,c)\oto(\bbY,d)$ by Proposition \ref{continuous_dist_condition}(ii), and the second equation is exactly $\Cd\zeta=f$.

{\bf Step 3.} $\Cd:(\QClsDist)^{\op}\lra\QSup$ is a quantaloid homomorphism. To show that $\Cd$ preserves joins of continuous $\CQ$-distributors, let $\{\zeta_i\}_{i\in I}\subseteq\QClsDist((\bbX,c),(\bbY,d))$ and one must check
$$\Big(\bv_{i\in I}\zeta_i\Big)^{\triangleright}=\bigsqcup_{i\in I}\zetra_i:\sC(\bbY,d)\lra\sC(\bbX,c),$$
where $\bigsqcup$ denotes the pointwise join in $\QSup(\sC(\bbY,d),\sC(\bbX,c))$ inherited from $\sC(\bbX,c)$. Indeed, since $\oc:\PX\lra\sC(\bbX,c)$ is a left adjoint $\CQ$-functor, one has
$$\Big(\bv_{i\in I}\zeta_i\Big)^{\triangleright}\lam=\oc\Big(\bv_{i\in I}\zeta_i\Big)^*\lam=\oc\Big(\lam\circ\bv_{i\in I}\zeta_i\Big)=\oc\Big(\bv_{i\in I}\lam\circ\zeta_i\Big)=\bigsqcup_{i\in I}\oc(\lam\circ\zeta_i)=\bigsqcup_{i\in I}\oc\zeta_i^*\lam=\bigsqcup_{i\in I}\zetra_i\lam$$
for all $\lam\in\sC(\bbY,d)$, where the fourth equality follows from Proposition \ref{la_tensor_preserving}, completing the proof.
\end{proof}

Let $\Id$ be the composite 2-functor
$$\Id:=(\QSup\to^{\sI}\QCatCls\to^{(-)^{\nat}}(\QClsDist)^{\op}).$$
Since $\sC\sI$ is naturally isomorphic to the identity 2-functor on $\QSup$ (see the first paragraph of the proof of Theorem \ref{C_I_adjoint}), one soon has:

\begin{prop} \label{Cd_Id_id}
$\Cd\Id$ is naturally isomorphic to the identity 2-functor on $\QSup$.
\end{prop}

\begin{proof}
Just note that $\Cd\Id=\Cd\cdot(-)^{\nat}\cdot\sI=\sC\sI$.
\end{proof}

\subsection{Closed continuous $\CQ$-distributors}

A \emph{nucleus} \cite{Rosenthal1996} on a quantaloid $\CQ$ is a lax functor $\sj:\CQ\lra\CQ$ that is an identity on objects and a closure operator on each hom-set. In elementary words, a nucleus $\sj$ consists of a family of monotone maps on each $\CQ(p,q)$ $(p,q\in\CQ_0)$ such that $u\leq\sj u$, $\sj\sj u=\sj u$ and $\sj v\circ\sj u\leq\sj(v\circ u)$ for all $u\in\CQ(p,q)$, $v\in\CQ(q,r)$.

Each nucleus $\sj:\CQ\lra\CQ$ induces a \emph{quotient quantaloid} $\CQ_{\sj}$ whose objects are the same as those of $\CQ$; arrows in $\CQ_{\sj}$ are the fixed points of $\sj$, i.e., $u\in\CQ_\sj(p,q)$ if $\sj u=u$ for $u\in\CQ(p,q)$. The identity arrow in $\CQ_\sj(q,q)$ is $\sj(1_q)$; local joins $\bigsqcup$ and compositions $\circ_{\sj}$ in $\CQ_\sj$ are respectively given by
\begin{equation} \label{Qj_comp_join}
\bigsqcup_{i\in I} u_i=\sj\Big(\bv_{i\in I} u_i\Big)\quad\text{and}\quad v\circ_\sj u=\sj(v\circ u)
\end{equation}
for $\{u_i\}_{i\in I}\subseteq\CQ_\sj(p,q)$, $u\in\CQ_\sj(p,q)$, $v\in\CQ_\sj(q,r)$. In addition, $\sj:\CQ\lra\CQ_\sj$ is a full quantaloid homomorphism.

\begin{rem}
$\CQ_{\sj}$ may be viewed as the quotient of $\CQ$ modulo the congruence $\vartheta_{\sj}$ (i.e., a family of equivalence relations $(\vartheta_{\sj})_{p,q}$ on each hom-set $\CQ(p,q)$ that is compatible with compositions and joins of $\CQ$-arrows) given by
$$(u,u')\in(\vartheta_{\sj})_{p,q}\iff\sj u=\sj u'.$$
In fact, $\sj u$ is the largest $\CQ$-arrow in the equivalence class of each $\CQ$-arrow $u$, thus $\CQ_{\sj}$ contains exactly one representative (i.e., the largest one) from each equivalence class of the congruence $\vartheta_{\sj}$.
\end{rem}

Recall that each $\CQ$-distributor $\phi:\bbX\oto\bbY$ has a transpose $\tphi:\bbY\lra\PX$ (see Equation (\ref{tphi_def})), and one may verify the following lemma easily:

\begin{lem} \label{tphi_hphi_Yoneda} {\rm\cite{Shen2013a}}
For each $\CQ$-distributor $\phi:\bbX\oto\bbY$ and $y\in\bbY_0$,
$$\tphi y=\phi(-,y)=\phi^*\sy_{\bbY}y.$$
\end{lem}

A continuous $\CQ$-distributor $\zeta:(\bbX,c)\oto(\bbY,d)$ is \emph{closed} if its transpose satisfies $\tzeta y\in\sC(\bbX,c)$ for all $y\in\bbY_0$. For a general $\zeta$, we define its \emph{closure} $\cl\zeta:\bbX\oto\bbY$ through its transpose as
$$\widetilde{\cl\zeta}:=(\bbY\to^{\tzeta}\PX\to^c\PX).$$

\begin{lem} \label{cpzeta=czeta}
Let $\zeta:(\bbX,c)\oto(\bbY,d)$ be a continuous $\CQ$-distributor. Then
\begin{itemize}
\item[\rm (1)] $c(\cl\zeta)^*=c\zeta^*$.
\item[\rm (2)] $\cl\zeta:(\bbX,c)\oto(\bbY,d)$ is a closed continuous $\CQ$-distributor.
\end{itemize}
\end{lem}

\begin{proof}
(1) Since $\oc:\PX\lra\sC(\bbX,c)$ is a left adjoint in $\QCat$, similar to Step 2 in the proof of Proposition \ref{Cd_functor} one has
\begin{align*}
c(\cl\zeta)^*\lam&=\oc(\lam\circ\cl\zeta)\\
&=\oc\Big(\bv_{y\in\bbY_0}\lam(y)\circ(\widetilde{\cl\zeta}y)\Big)\\
&=\oc\Big(\bv_{y\in\bbY_0}\lam(y)\circ(\oc\tzeta y)\Big)\\
&=\bigsqcup_{y\in\bbY_0}\oc(\lam(y)\circ(\oc\tzeta y))&(\text{Proposition \ref{la_tensor_preserving}})\\
&=\bigsqcup_{y\in\bbY_0}\lam(y)\otimes(\oc\tzeta y)&(\text{Equation (\ref{tensor_cPX})})\\
&=\oc\Big(\bv_{y\in\bbY_0}\lam(y)\circ(\tzeta y)\Big)&(\text{Proposition \ref{la_tensor_preserving}})\\
&=\oc(\lam\circ\zeta)\\
&=c\zeta^*\lam.
\end{align*}
for all $\lam\in\PY$, where $\bv$ and $\bigsqcup$ respectively denote the underlying joins in $\PX$ and $\sC(\bbX,c)$, and $\otimes$ denotes the tensor in $\sC(\bbX,c)$.

(2) Proposition \ref{continuous_dist_condition}(ii) together with (1) ensure that
$$c(\cl\zeta)^*d=c\zeta^*d=c\zeta^*=c(\cl\zeta)^*,$$
and hence, $\cl\zeta:(\bbX,c)\oto(\bbY,d)$ is a continuous $\CQ$-distributor which is obviously closed.
\end{proof}

\begin{prop} \label{p_nucleus}
$\cl$ is a nucleus on the quantaloid $\QClsDist$.
\end{prop}

\begin{proof}
First, it is easy to check that $\cl$ is monotone with respect to the local order of continuous $\CQ$-distributors and $\zeta\leq\cl\zeta$, $\cl\cdot\cl\zeta=\cl\zeta$.

Second, in order to prove
$$\cl\eta\circ\cl\zeta\leq\cl(\eta\circ\zeta)$$
for all continuous $\CQ$-distributors $\zeta:(\bbX,c)\oto(\bbY,d)$, $\eta:(\bbY,d)\oto(\bbZ,e)$, note that
\begin{align*}
\widetilde{\cl\eta\circ\cl\zeta}&\leq c\cdot\widetilde{\cl\eta\circ\cl\zeta}\\
&=c(\cl\eta\circ\cl\zeta)^*\sy_{\bbZ}&(\text{Lemma \ref{tphi_hphi_Yoneda}})\\
&=c(\cl\zeta)^*(\cl\eta)^*\sy_{\bbZ}\\
&=c(\cl\zeta)^*d(\cl\eta)^*\sy_{\bbZ}&(\text{Proposition \ref{continuous_dist_condition}(ii)})\\
&=c\zeta^*d\eta^*\sy_{\bbZ}&(\text{Lemma \ref{cpzeta=czeta}(1)})\\
&=c\zeta^*\eta^*\sy_{\bbZ}&(\text{Proposition \ref{continuous_dist_condition}(ii)})\\
&=c(\eta\circ\zeta)^*\sy_{\bbZ}\\
&=c\cdot\widetilde{\eta\circ\zeta}&(\text{Lemma \ref{tphi_hphi_Yoneda}})\\
&=\widetilde{\cl(\eta\circ\zeta)},
\end{align*}
and the conclusion thus follows.
\end{proof}

The nucleus $\cl$ gives rise to a quotient quantaloid of $\QClsDist$, i.e.,
$$\QClsCloDist.$$
We remind the readers that local joins and compositions in the quantaloid $\QClsCloDist$ of $\CQ$-closure spaces and closed continuous $\CQ$-distributors are given by the formulas in (\ref{Qj_comp_join}), which are in general different from those in $\QClsDist$.

The universal property of the quotient quantaloid $\QClsCloDist$ along with the following Lemma \ref{Czeta=Ceta} ensures that $\Cd$ factors uniquely through the quotient homomorphism $\cl$ via a quantaloid homomorphism $\Ccl$:
$$\bfig
\qtriangle/->`->`-->/<1800,500>[(\QClsDist)^{\op} `\QClsCloDist^{\op}`\QSup;\cl^{\op}`\Cd`\Ccl]
\efig$$

\begin{lem} \label{Czeta=Ceta}
For continuous $\CQ$-distributors $\zeta,\eta:(\bbX,c)\oto(\bbY,d)$, $\cl\zeta=\cl\eta$ if, and only if, $\Cd\zeta=\Cd\eta$.
\end{lem}

\begin{proof}
The necessity is easy, since from Lemma \ref{cpzeta=czeta}(1) one soon has
$$\oc\zeta^*=\oc(\cl\zeta)^*=\oc(\cl\eta)^*=\oc\eta^*:\PY\lra\sC(\bbX,c),$$
and thus
$$\Cd\zeta=\zetra=\etra=\Cd\eta: \sC(\bbY,d)\ \to/^(->/\PY\to^{\oc\zeta^*=\oc\eta^*}\sC(\bbX,c).$$

For the sufficiency, if $\zetra=\etra$, it follows immediately that
$$\oc\zeta^*d=\zetra\od=\etra\od=\oc\eta^*d:\PY\lra\sC(\bbX,c),$$
and consequently
$$\widetilde{\cl\zeta}=c\tzeta=c\zeta^*\sy_{\bbY}=c\zeta^*d\sy_{\bbY}=c\eta^*d\sy_{\bbY}=c\eta^*\sy_{\bbY}=c\teta=\widetilde{\cl\eta},$$
where Lemma \ref{tphi_hphi_Yoneda} implies the second and the sixth equalities, and Proposition \ref{continuous_dist_condition}(ii) guarantees the third and the fifth equalities. Therefore $\cl\zeta=\cl\eta$.
\end{proof}

Let $\Icl$ be the composite 2-functor
$$\Icl:=(\QSup\to^{\Id}(\QClsDist)^{\op}\to^{\cl^{\op}}\QClsCloDist^{\op}).$$
Then $\Ccl\Icl=\Ccl\cdot\cl^{\op}\cdot\Id=\Cd\Id$, and together with Proposition \ref{Cd_Id_id} one has:

\begin{prop} \label{Ccl_Icl_id}
$\Ccl\Icl$ is naturally isomorphic to the identity 2-functor on $\QSup$.
\end{prop}

Note that Proposition \ref{Cd_functor} and Lemma \ref{Czeta=Ceta} guarantee that $\Ccl$ is fully faithful, while Proposition \ref{Ccl_Icl_id} in particular implies that $\Ccl$ is essentially surjective. Therefore, we arrive at the main result of this paper:

\begin{thm} \label{Ccl_Icl_equivalence}
$\Ccl:\QClsCloDist^{\op}\lra\QSup$ and $\Icl:\QSup\lra\QClsCloDist^{\op}$ establish an equivalence of quantaloids; hence, $\QClsCloDist$ and $\QSup$ are dually equivalent quantaloids.
\end{thm}

\begin{proof}
It remains to verify the claim about $\Icl$. First, since $\Ccl$ is an equivalence of categories, there exists a functor $\sF:\QSup\lra\QClsCloDist^{\op}$ such that $\sF\Ccl$ is naturally isomorphic to the identity functor on $\QClsCloDist^{\op}$, thus so is $\Icl\Ccl$ as one has natural isomorphisms
$$\Icl\Ccl\cong\sF\Ccl\Icl\Ccl\cong\sF\Ccl,$$
showing that $\Icl$ is also an equivalence of categories. Second, $\Icl$ is a quantaloid homomorphism since it is fully faithful and clearly preserves the order of hom-sets, and consequently preserves joins of left-adjoint $\CQ$-functors.
\end{proof}

\begin{rem}
In fact, for any left adjoint $\CQ$-functor $f:\bbX\lra\bbY$ between complete $\CQ$-categories, $\Id f=f^{\nat}:(\bbY,c_{\bbY})\oto(\bbX,c_{\bbX})$ is a closed continuous $\CQ$-distributor, since
$$\widetilde{f^{\nat}}x=f^{\nat}(-,x)=\sy_{\bbY}(fx)\in\sC(\bbY,c_{\bbY})$$
for all $x\in\bbX_0$. That is, $\Id f=\Icl f$.
\end{rem}

It is already known that $\QSup$ is monadic over $\Set\da\CQ_0$ (see \cite[Theorem 3.8]{Pu2015}), thus $\QSup$ is complete since so is $\Set\da\CQ_0$ (see \cite[Corollary II.3.3.2]{Hofmann2014}). Moreover, the 2-isomorphism (\ref{Qop_Cat_iso}) in Remark \ref{QCat_dual} induces an isomorphism of quantaloids
$$(\QSup)^{\op}\cong\CQ^{\op}\text{-}\Sup,$$
which assigns to a left adjoint $\CQ$-functor $f:\bbX\lra\bbY$ the dual $g^{\op}:\bbY^{\op}\lra\bbX^{\op}$ of its right adjoint. Hence, the completeness of $\CQ^{\op}\text{-}\Sup$ guarantees the cocompleteness of $\QSup$, and in combination with Theorem \ref{Ccl_Icl_equivalence} one concludes:

\begin{cor} \label{QCatClsCloDist_complete}
$\QClsCloDist$ is cocomplete and complete.
\end{cor}

\subsection{Continuous $\CQ$-relations}

Let
$$\QClsRel\quad(\text{resp.}\ \QClsCloRel)$$
denote the full subquantaloid of $\QClsDist$ (resp. $\QClsCloDist$) whose objects are $\CQ$-closure spaces with discrete underlying $\CQ$-categories. Morphisms in $\QClsRel$ (resp. $\QClsCloRel$) will be called \emph{continuous $\CQ$-relations} (resp. \emph{closed continuous $\CQ$-relations}).

The following conclusion follows soon from Proposition \ref{Ud_topological}, where one only needs to replace all $\CQ$-categories in its proof with discrete ones:

\begin{prop}
The forgetful functor $\QClsRel\lra\QRel$ is topological. In particular, the category $\ClsRel$ of closure spaces and continuous relations is topological over the category $\Rel$ of sets and relations.
\end{prop}

As a full subquantaloid of $\QClsCloDist$, $\QClsCloRel$ is also dually equivalent to $\QSup$:

\begin{thm} \label{QClsCloRel_QSup_equivalent}
$\QClsCloDist$ is equivalent to its full subquantaloid $\QClsCloRel$. Thus one has equivalences of quantaloids
$$\QClsCloRel^{\op}\simeq\QClsCloDist^{\op}\simeq\QSup.$$
In particular, $\QClsCloRel$ is cocomplete and complete.
\end{thm}

\begin{proof}
It suffices to show that each $\CQ$-closure space $(\bbX,c)$ is isomorphic to $(\bbX_0,c_0)$ in the category $\QClsCloDist$. For this, note that $\Ccl(\bbX,c)=\sC(\bbX,c)=\sC(\bbX_0,c_0)=\Ccl(\bbX_0,c_0)$ by Proposition \ref{cPX=c0PX0}, and $\Icl\Ccl$ is naturally isomorphic to the identity functor on $\QClsCloDist^{\op}$ by Theorem \ref{Ccl_Icl_equivalence}. Therefore
$$(\bbX,c)\cong\Icl\Ccl(\bbX,c)=\Icl\Ccl(\bbX_0,c_0)\cong(\bbX_0,c_0)$$
in the category $\QClsCloDist$, as desired.
\end{proof}

In particular, for the case $\CQ={\bf 2}$, $\Sup$ is monadic over $\Set$, and it is known in category theory that
\begin{itemize}
\item for a \emph{solid} (=\emph{semi-topological} \cite{Tholen1979}) functor $\CE\lra\CB$, if $\CB$ is totally cocomplete, then so is $\CE$ \cite{Tholen1980};
\item every monadic functor over $\Set$ is solid (see \cite[Example 4.4]{Tholen1979});
\item $\Sup$ is a self-dual category, i.e., $\Sup\cong\Sup^{\op}$.
\end{itemize}
Thus we conclude:

\begin{cor} \label{ClsCloRel_Sup_equivalent}
The quantaloid $\ClsCloRel$ of closure spaces and closed continuous relations is equivalent to the quantaloid $\Sup$ of complete lattices and $\sup$-preserving maps. Therefore, $\ClsCloRel$ is totally cocomplete and totally complete and, in particular, cocomplete and complete.
\end{cor}


\section{Example: Fuzzy closure spaces on fuzzy sets} \label{Example_Fuzzy_closure_spaces}

Based on the characterizations of fuzzy sets, fuzzy preorders and fuzzy powersets as quantaloid-enriched categories \cite{Hohle2011,Pu2012,Shen2013b,Tao2014}, in this section we introduce fuzzy closure spaces on fuzzy sets as an example of $\CQ$-closure spaces.

\subsection{Preordered fuzzy sets valued in a divisible quantale}

A quantaloid with only one object is a \emph{unital quantale} \cite{Rosenthal1990}. With $\&$ denoting the multiplication in a quantale $Q$ (i.e., the composition in the unique hom-set of the one-object quantaloid), one has the \emph{implications} $/$, $\bs$ in $Q$ (i.e., the left and right implications in the one-object quantaloid) determined by the adjoint property
$$x\& y\leq z\iff x\leq z/y\iff y\leq x\bs z\quad(x,y,z\in Q).$$

A unital quantale $(Q,\&)$ is \emph{divisible} \cite{Hohle1995a} if it satisfies one of the equivalent conditions in the following proposition:

\begin{prop} \label{divisible_condition} {\rm\cite{Pu2012,Tao2014}}
For a unital quantale $(Q,\&)$, the following conditions are equivalent:
\begin{itemize}
\item[\rm (i)] $\forall x,y\in Q$, $x\leq y$ implies $y\&a=x=b\&y$ for some $a,b\in Q$.
\item[\rm (ii)] $\forall x,y\in Q$, $x\leq y$ implies $y\&(y\bs x)=x=(x/y)\&y$.
\item[\rm (iii)] $\forall x,y,z\in Q$, $x,y\leq z$ implies $x\&(z\bs y)=(x/z)\&y$.
\item[\rm (iv)] $\forall x,y\in Q$, $x\&(x\bs y)=x\wedge y=(y/x)\&x$.
\end{itemize}
In this case, the unit of the quantale $(Q,\&)$ must be the top element of $Q$.
\end{prop}

Divisible unital quantales cover most of the important truth tables in fuzzy set theory:

\begin{exmp} \phantomsection \label{divisible_quantale}
\begin{itemize}
\item[\rm (1)] Each frame is a divisible unital quantale.
\item[\rm (2)] Each complete BL-algebra \cite{Hajek1998} is a divisible unital quantale. In particular, the unit interval $[0,1]$ equipped with a continuous t-norm \cite{Klement2000} is a divisible unital quantale.
\item[\rm (3)] The extended real line $([0,\infty]^{\op},+)$ \cite{Lawvere1973} is a divisible unital quantale in which $y/x=x\bs y=\max\{0,y-x\}$.\footnote{To avoid confusion, the symbols $\max$, $\vee$, $\leq$, etc. between (extended) real numbers always refer to the standard order, although the quantale $([0,\infty]^{\op},+)$ is equipped with the reverse order of real numbers.}
\end{itemize}
\end{exmp}

Throughout this section, $Q$ always denotes a divisible unital quantale with the multiplication $\&$ and implications $/$, $\bs$ unless otherwise specified. The top and bottom elements in $Q$ are respectively $1$ and $0$ (note that in a divisible unital quantale $Q$, $1$ must be the unit for the multiplication $\&$ by Proposition \ref{divisible_condition}).

Being considered as a one-object quantaloid, $Q$-categories $\bbX=(X,\al)$ are widely known as \emph{crisp} sets $X$ equipped with \emph{fuzzy preorder} $\al:X\times X\lra Q$. In particular:

\begin{exmp}
\begin{itemize}
\item[\rm (1)] For the two-element Boolean algebra ${\bf 2}$, ${\bf 2}$-categories are just preordered sets.
\item[\rm (2)] If $(Q,\&)=([0,\infty]^{\op},+)$, then $Q$-categories are (generalized) \emph{metric spaces} \cite{Lawvere1973}; that is, sets $X$ carrying distance functions $a:X\times X\lra[0,\infty]$ satisfying $a(x,x)=0$ and $a(x,z)\leq a(x,y)+a(y,z)$ for all $x,y,z\in X$.
\end{itemize}
\end{exmp}

%

In order to characterize fuzzy preorder on \emph{fuzzy} sets, the following quantaloid $\DQ$ is crucial:

\begin{prop} \label{divisible_quantale_quantaloid} {\rm\cite{Hohle2011,Pu2012}}
For a divisible unital quantale $Q$, the following data define a quantaloid $\DQ$:
\begin{itemize}
\item $\ob(\DQ)=Q$;
\item $\DQ(x,y)=\{u\in Q:u\leq x\wedge y\}$ with inherited order from $Q$;\footnote{The readers should be cautious that different hom-sets in $\DQ$ are considered to be disjoint; that is, a $\DQ$-arrow $u\in\DQ(x,y)$ is a triple $(u,x,y)$ rather than an element $u$ of $Q$.}
\item the composition of $\DQ$-arrows $u\in\DQ(x,y)$, $v\in\DQ(y,z)$ is given by
    $$v\circ u=v\&(y\bs u)=(v/y)\&u;$$
\item the implications of $\DQ$-arrows are given by
    $$w\lda u=y\wedge z\wedge(w/(y\bs u))\quad\text{and}\quad v\rda w=x\wedge y\wedge((v/y)\bs w)$$
    for all $u\in\DQ(x,y)$, $v\in\DQ(y,z)$, $w\in\DQ(x,z)$;
\item the identity $\DQ$-arrow on $x$ is $x$ itself.
\end{itemize}
\end{prop}

\begin{exmp} \phantomsection \label{DQ_exmp}
\begin{itemize}
\item[\rm (1)] For the two-element Boolean algebra ${\bf 2}=\{0,1\}$, $\CD{\bf 2}(1,1)$ contains two arrows: $0$ and $1$, and $0$ is the only arrow in every other hom-set.
\item[\rm (2)] If $(Q,\&)=([0,\infty]^{\op},+)$, then $\DQ(x,y)=\ua(x\vee y)$, i.e., the upper set generated by $x\vee y$. The composition of $u\in\DQ(x,y)$, $v\in\DQ(y,z)$ is
$v\circ u=v+u-y$.
\end{itemize}
\end{exmp}

A $Q$-typed set (or equivalently, a $(\DQ)_0$-typed set) is exactly a crisp set $X$ equipped with a map $m:X\lra Q$; that is, a \emph{fuzzy set} \cite{Zadeh1965}. $(X,m)$ is also called a \emph{$Q$-subset} of $X$, where the value $mx$ is the membership degree of $x$ in $(X,m)$. Fuzzy sets and membership-preserving maps constitute the slice category $\Set\da Q$.

A $\DQ$-relation $\phi:(X,m_X)\oto(Y,m_Y)$ is a \emph{fuzzy relation} between fuzzy sets $(X,m_X)$ and $(Y,m_Y)$, which is a map $X\times Y\lra Q$ satisfying
\begin{equation} \label{fuzzy_relation}
\phi(x,y)\leq m_X x\wedge m_Y y
\end{equation}
for all $x\in X$ and $y\in Y$. With the value $\phi(x,y)$ interpreted as the degree of $x$ and $y$ being related, Equation (\ref{fuzzy_relation}) asserts that the degree of $x$ and $y$ being related cannot exceed the membership degree of $x$ in $X$ or that of $y$ in $Y$.

A $\DQ$-category $\bbX=(X,m,\al)$ is exactly a \emph{fuzzy set} $(X,m)$ equipped with a \emph{fuzzy preorder} $\al$ (or, \emph{preordered fuzzy set} for short) \cite{Hohle2011,Pu2012}. In elementary words, $\al:X\times X\lra Q$ is a map satisfying
\begin{itemize}
\item $\al(x,y)\leq mx\wedge my$,
\item $mx\leq\al(x,x)$,
\item $\al(y,z)\&(my\bs\al(x,y))=(\al(y,z)/my)\&\al(x,y)\leq\al(x,z)$
\end{itemize}
for all $x,y,z\in X$. Note that the first and the second conditions together lead to $mx=\al(x,x)$ for all $x\in X$, thus a preordered fuzzy set may be described by a pair $(X,\al)$, where $X$ is a crisp set and $\al:X\times X\lra Q$ is a map, such that
\begin{itemize}
\item $\al(x,y)\leq\al(x,x)\wedge\al(y,y)$,
\item $\al(y,z)\&(\al(y,y)\bs\al(x,y))=(\al(y,z)/\al(y,y))\&\al(x,y)\leq\al(x,z)$
\end{itemize}
for all $x,y,z\in X$.

A $\DQ$-functor $f:(X,\al)\lra(Y,\be)$ is a \emph{monotone} map between preordered fuzzy sets, which is a map $f:X\lra Y$ satisfying
$$\al(x,x)=\be(fx,fx)\quad\text{and}\quad\al(x,x')\leq\be(fx,fx')$$
for all $x,x'\in X$. Preordered fuzzy sets and monotone maps constitute the category $\DQ\text{-}\Cat$.

\begin{exmp} \phantomsection \label{Q_preordered_set}
\begin{itemize}
\item[\rm (1)] A $\CD{\bf 2}$-category $(X,\al)$ is a ``partially defined'' preordered set; that is, a subset $A\subseteq X$ consisting of all those elements $x\in X$ with $\al(x,x)=1$ and a preorder on $A$. A $\CD{\bf 2}$-functor $f:(X,A)\lra(Y,B)$ is a map $f:X\lra Y$ monotone on $A=f^{\la}(B)$.
\item[\rm (2)] If $(Q,\&)=([0,\infty]^{\op},+)$, then $\DQ$-categories are (generalized) \emph{partial metric spaces}\footnote{The term ``partial metric'' was originally introduced by Matthews \cite{Matthews1994} with additional requirements of finiteness ($a(x,y)<\infty$), symmetry ($a(x,y)=a(y,x)$) and separatedness ($a(x,x)=a(x,y)=a(y,y)\iff x=y$) which are dropped here.} \cite{Hohle2011,Pu2012}; that is, sets $X$ carrying distance functions $a:X\times X\lra[0,\infty]$ satisfying $a(x,x)\vee a(y,y)\leq a(x,y)$ and $a(x,z)\leq a(x,y)+a(y,z)-a(y,y)$ for all $x,y,z\in X$. $\DQ$-functors $f:(X,a)\lra(Y,b)$ are non-expanding maps satisfying $a(x,x)=b(fx,fx)$ for all $x\in X$.
\end{itemize}
\end{exmp}

It is obvious that every $Q$-category $(X,\al)$ is a \emph{global} $\DQ$-category in the sense that $\al(x,x)=1$ for all $x\in X$; in fact, $Q\text{-}\Cat$ is a coreflective subcategory of $\DQ\text{-}\Cat$. So, crisp sets equipped with fuzzy preorder are a special case of preordered fuzzy sets; for example, partial metric spaces $(X,a)$ in which $a(x,x)=0$ for all $x\in X$ are exactly metric spaces.

\subsection{Fuzzy powersets of fuzzy sets}

For a preordered fuzzy set $(X,\al)$ (with underlying fuzzy set $(X,m)$), its $\DQ$-category of presheaves $(\sP(X,\al),S_{(X,\al)})$ is again a preordered fuzzy set, whose underlying fuzzy set $(\sP(X,\al),M)$ is the \emph{fuzzy set of lower fuzzy subsets} of $(X,\al)$ \cite{Shen2013b}:
\begin{itemize}
\item A fuzzy set $(X,n)$ is a \emph{fuzzy subset} of $(X,m)$ if $nx\leq mx$ for all $x\in X$; that is, the membership degree of $x$ in $(X,n)$ does not exceed that of $x$ in $(X,m)$.
\item A \emph{lower fuzzy subset} of $(X,\al)$ is a fuzzy subset $(X,l)$ of $(X,m)$ satisfying
    $$ly\&(my\bs\al(x,y))=(ly/my)\&\al(x,y)\leq lx$$
    for all $x,y\in X$, which intuitively means that $y$ is in $(X,l)$ and $x$ is less than or equal to $y$ implies $x$ is in $(X,l)$.
\item A \emph{potential lower fuzzy subset} of $(X,\al)$ is a triple $(X,l,q)$, where $(X,l)$ is a lower fuzzy subset of $(X,\al)$ and $q\in\CQ_0$, such that $lx\leq q$ for all $x\in X$. Thus $(X,l,q)$ satisfies
    $$ly\&(my\bs\al(x,y))=(ly/my)\&\al(x,y)\leq lx\leq mx\wedge q$$
    for all $x,y\in X$. In other words, potential lower fuzzy subsets $(X,l,q)$ of $(X,\al)$ are exactly $\DQ$-distributors $(X,\al)\oto\{q\}$.
\item $\sP(X,\al)$ is a crisp set whose elements are all the potential lower fuzzy subsets of $(X,\al)$. As the fuzzy set of lower fuzzy subsets of $(X,\al)$, $(\sP(X,\al),M)$ is a fuzzy set (i.e., a $Q$-subset of $\sP(X,\al)$) with the membership degree map $M:\sP(X,\al)\lra Q$ given by
    $$M(X,l,q)=q,$$
    which gives the degree of $(X,l,q)$ being a lower fuzzy subset of $(X,\al)$.
\end{itemize}

The separated preorder $S_{(X,\al)}$ on $(\sP(X,\al),M)$ is given by
\begin{equation} \label{SXal}
S_{(X,\al)}((X,l,q),(X,l',q'))=q\wedge q'\wedge\bw_{x\in X}l'x/(q\bs lx)
\end{equation}
for all $(X,l,q),(X,l',q')\in\sP(X,\al)$, which is intuitively the inclusion order of potential lower fuzzy subsets.

Dually, the $\DQ$-category of copresheaves on $(X,\al)$ is the preordered \emph{fuzzy set of upper fuzzy subsets} of $(X,\al)$ and we do not bother spell out the details.

\begin{exmp}[Fuzzy powersets] \label{fuzzy_powerset}
For each fuzzy set $(X,m)$, the \emph{fuzzy powerset} of $(X,m)$ \cite{Shen2013b} is defined as
$$(\sP(X,m),M):=(\sP(X,\id_{(X,m)}),M).$$
Explicitly, elements in $\sP(X,m)$ are \emph{potential fuzzy subsets} $(X,n,q)$ of $(X,m)$ that satisfy
\begin{equation} \label{potential_fuzzy_subset}
nx\leq mx\wedge q
\end{equation}
for all $x\in X$; or equivalently, fuzzy relations $(X,m)\oto\{q\}$. 
It should be reminded that, although $(X,m)$ is a discrete $\DQ$-category, $(\sP(X,m),S_{(X,m)})$ is \emph{not} discrete, whose structure relies on that of $\DQ$.

We point out that for a crisp set $X$, the fuzzy powerset $(\sP X,M)$ of $X$ is different from the \emph{crisp set} $Q^X$ of maps from $X$ to $Q$, which is referred to as the \emph{$Q$-powerset} of $X$ (also called the fuzzy powerset of $X$ by some authors) in the literature:
\begin{itemize}
\item $Q^X$ is a \emph{crisp set} that consists of \emph{fuzzy subsets} $(X,n)$ of $X$;
\item $(\sP X,M)$ is a \emph{fuzzy set} whose underlying crisp set $\sP X$ consists of \emph{potential fuzzy subsets} $(X,n,q)$ of $X$.
\item From the viewpoint of category theory, $Q^X$ is the underlying set of the presheaf $Q$-category of the discrete $Q$-category $X$, while $(\sP X,M)$ is the underlying $(\DQ)_0$-typed set of the presheaf $\DQ$-category of the discrete $\DQ$-category $X$.
\end{itemize}

\end{exmp}

\begin{exmp}
For a partial metric space $(X,a)$, a presheaf on $(X,a)$ is a pair $(h,r)$, where $h:X\lra[0,\infty]$ is a map and $r\in[0,\infty]$, such that
$$a(x,x)\vee r\leq hx\leq a(x,y)+hy-a(y,y)$$
for all $x,y\in X$; dually, such a pair $(h,r)$ satisfying
$$a(x,x)\vee r\leq hx\leq a(y,x)+hy-a(y,y)$$
for all $x,y\in X$ is a copresheaf on $(X,a)$. 
\end{exmp}

%

A preordered fuzzy set $(X,m,\al)$ is complete if every potential lower fuzzy subset $(X,l,q)\in\sP(X,\al)$ has a supremum given by an element $s\in X$ with membership degree $ms=q$, such that
$$\al(s,x)=q\wedge mx\wedge\bw_{y\in X}\al(y,x)/(q\bs ly)$$
for all $x\in X$. One may translate the above equation as: $s$ is less than or equal to $x$ if, and only if, each $y$ in $(X,l,q)$ is less than or equal to $x$; furthermore, $ms=q$ indicates that the degree of $(X,l,q)$ being a lower fuzzy subset of $(X,m,\al)$ is equal to the membership degree of its supremum, if exists, in $(X,m)$. 

Separated complete preordered fuzzy sets and $\sup$-preserving maps constitute the quantaloid $\DQ\text{-}\Sup$.

\begin{exmp}
\begin{itemize}
\item[\rm (1)] A global separated complete preordered fuzzy set, i.e., an object in $Q\text{-}\Sup$, is also called a \emph{complete $Q$-lattice} \cite{Shen2013}. However, although $Q\text{-}\Cat$ is a coreflective subcategory of $\DQ\text{-}\Cat$ and the coreflector sends each separated complete preordered fuzzy set to a complete $Q$-lattice, it should be cautious that no object of $Q\text{-}\Sup$ lies in $\DQ\text{-}\Sup$ as long as $Q$ has more than one elements; in fact, a complete preordered fuzzy set valued in a non-trivial quantale $Q$ can never be global as Remark \ref{complete_nonempty} indicates.
\item[\rm (2)] A partial metric space $(X,a)$ is complete if for every presheaf $(h,r)$ on $(X,a)$, there is $s\in X$ such that
$$a(s,s)=r\quad\text{and}\quad a(s,x)=r\vee a(x,x)\vee\bv_{y\in X}(a(y,x)-hy+r)$$
for all $x\in X$. As a comparison, a metric space $(X,a)$ is complete\footnote{The completeness of metric spaces discussed here is in the categorical sense which is stronger than the classical Cauchy completeness.} if for every map $h:X\lra[0,\infty]$ satisfying
$$\forall x,y\in X:\ hx-hy\leq a(x,y),$$
there exists $s\in X$ with
$$a(s,x)=0\vee\bv_{y\in X}(a(y,x)-hy)$$
for all $x\in X$. As we point out in (1), the coreflector from $\CD[0,\infty]^{\op}\text{-}\Cat$ to $[0,\infty]^{\op}\text{-}\Cat$ sends each complete partial metric space to a complete metric space, but a metric space can never be complete when it is considered as a global partial metric space.
\end{itemize}
\end{exmp}

\subsection{Fuzzy closure spaces on fuzzy sets} \label{Fuzzy_closure_spaces}

By a \emph{fuzzy closure space} on a fuzzy set we mean an object in the category $\DQ\text{-}\Cls$. Explicitly, a fuzzy closure space is a triple $(X,m,c)$, where $(X,m)$ is a fuzzy set, $c:\sP(X,m)\lra\sP(X,m)$ is a map such that for any potential fuzzy subsets $(X,n,q),(X,n',q')\in\sP(X,m)$,
\begin{itemize}
\item $M(c(X,n,q))=q$, where $(\sP(X,m),M)$ is the fuzzy powerset of $(X,m)$,
\item $S_{(X,m)}((X,n,q),(X,n',q'))\leq S_{(X,m)}(c(X,n,q),c(X,n',q'))$,
\item $(X,n,q)\leq c(X,n,q)$, and
\item $cc(X,n,q)=c(X,n,q)$.
\end{itemize}

As we mentioned for general $\CQ$-closure spaces, a fuzzy closure space $(X,m,c)$ may be equivalently described by the fixed points of $c$; or, one may call them \emph{potential closed fuzzy subsets} of $(X,m)$. To explain this term, first note that a potential fuzzy subset $(X,n,q)$ of $(X,m)$ is \emph{closed} if $c(X,n,q)=(X,n,q)$; next, putting the potential closed fuzzy subsets of $(X,m)$ together, one again obtains a \emph{fuzzy set} $(\sC(X,m,c),\overline{M})$, where $\overline{M}:\sC(X,m,c)\lra Q$ is the restriction of $M:\sP(X,m)\lra Q$ on $\sC(X,m,c)$. Then elements in $\sC(X,m,c)$ are \emph{potential closed fuzzy subsets} in the sense that, each potential fuzzy subset $(X,n,q)\in\sC(X,m,c)$ is closed, and $\overline{M}(X,n,q)=q$ gives the degree of $(X,n,q)$ being a closed fuzzy subset of $(X,m)$. That is to say:
\begin{quote}
A fuzzy closure space is determined by a \emph{fuzzy set} $(X,m)$ and a \emph{fuzzy set of closed fuzzy subsets} of $(X,m)$.
\end{quote}

A continuous map $f:(X,m_X,c)\lra(Y,m_Y,d)$ between fuzzy closure spaces is a membership-preserving map $f:(X,m_X)\lra(Y,m_Y)$ such that the inverse images of potential closed fuzzy subsets of $(Y,m_Y,d)$ are closed in $(X,m_X,c)$ (see Proposition \ref{continuous_functor_condition}(iv)). The following corollary is a natural generalization of the well-known fact that the category $\Cls$ of closure spaces and continuous maps is topological over $\Set$:

\begin{cor} \label{DQCls_topological}
The category $\DQ\text{-}\Cls$ of fuzzy closure spaces and continuous maps is topological over the category $\Set\da Q$ of fuzzy sets and membership-preserving maps.
\end{cor}

The readers should carefully distinguish fuzzy closure spaces defined here from $Q$-closure spaces (also called ``fuzzy closure spaces'' by some authors) in the existing literature: a $Q$-closure space is given by a \emph{crisp set} $X$ and a \emph{crisp set} of closed $Q$-subsets of $X$; the category of $Q$-closure spaces and continuous maps is $Q\text{-}\Cls$, where $Q$ is considered as a one-object quantaloid. As a comparison to Corollary \ref{DQCls_topological}, note that $Q\text{-}\Cls$ is topological over $\Set$, since the underlying sets of $Q$-closure spaces are crisp.

\begin{exmp}[Specialization preorder of fuzzy closure spaces]
For a fuzzy closure space $(X,m,c)$, $\al:=\tc\rda\tc$ defines the specialization preorder (see Subsection \ref{Specialization_QCat}) on the fuzzy set $(X,m)$ given by
$$\al(x,y)=mx\wedge my\wedge\bw_{(X,n,q)\in\sC(X,m,c)}(ny/my)\bs nx$$
for all $x,y\in X$, which extends the notion of the specialization order of fuzzy topological spaces (on crisp sets) in \cite{Lai2006}.
\end{exmp}


\begin{exmp}[Alexandrov spaces on fuzzy sets]
A $\DQ$-Alexandrov space is a fuzzy set $(X,m)$ equipped with a family of potential fuzzy subsets of $(X,m)$ that is closed with respect to underlying joins, underlying meets, tensors and cotensors in $(\sP(X,m),S_{(X,m)})$. Thus, $\DQ$-Alexandrov spaces are in fact a special kind of \emph{fuzzy topological spaces on fuzzy sets} (a notion that we will try to clarify in future works): recall that, classically, an Alexandrov space is a topological space in which arbitrary joins and arbitrary meets of open subsets are still open.
\end{exmp}


A fuzzy relation $\zeta:(X,m_X,c)\oto(Y,m_Y,d)$ between fuzzy closure spaces is continuous if
$$\zeta_*:(\sP(X,m_X),S_{(X,m_X)})\lra(\sP(Y,m_Y),S_{(Y,m_Y)})$$
sends each potential closed fuzzy subset of $(X,m_X,c)$ to a potential closed fuzzy subset of $(Y,m_Y,d)$ (see Proposition \ref{continuous_dist_condition}(iv)). $\zeta$ is moreover closed if $(X,\zeta(-,y),m_Y y)$ is a potential closed fuzzy subset of $(X,m_X,c)$ for all $y\in Y$. From Theorem \ref{QClsCloRel_QSup_equivalent} one immediately has:

\begin{cor}
The quantaloid $(\DQ\text{-}\ClsRel)_{\cl}$ of fuzzy closure spaces and closed continuous fuzzy relations is dually equivalent to the quantaloid $\DQ\text{-}\Sup$ of separated complete preordered fuzzy sets and $\sup$-preserving maps.
\end{cor}

In the case that $Q$ is a \emph{commutative} unital quantale, it is not difficult to see that $Q\text{-}\Sup$ is self-dual, and thus Theorem \ref{QClsCloRel_QSup_equivalent} reduces to:

\begin{cor}
For a commutative unital quantale $Q$, the quantaloid $(Q\text{-}\ClsRel)_{\cl}$ of $Q$-closure spaces and closed continuous $Q$-relations is equivalent to the quantaloid $Q\text{-}\Sup$ of complete $Q$-lattices and $\sup$-preserving maps.
\end{cor}

\section{Conclusion}

The following diagram summarizes the pivotal categories and functors treated in this paper:
$$\bfig
\square|arla|/->`@{^(->}@<-5pt>`@{^(->}@<-5pt>`->/<1000,500>[\QCls`(\QClsRel)^{\op}`\QCatCls`(\QClsDist)^{\op};(-)^{\nat}`\dv``(-)^{\nat}]
\square(1000,0)|arra|/->`@{^(->}@<-5pt>`@{^(->}@<-2pt>`->/<1200,500>[(\QClsRel)^{\op}`\QClsCloRel^{\op}`(\QClsDist)^{\op}`\QClsCloDist^{\op};\cl^{\op}``\simeq`\cl^{\op}]
\morphism/@{->}@<-6pt>/<0,500>[\QCatCls`\QCls;(-)_0]
\morphism(1000,0)|r|/@{->}@<-5pt>/<0,500>[(\QClsDist)^{\op}`(\QClsRel)^{\op};(-)_0^{\op}]
\morphism(2200,0)|r|/@{->}@<-10pt>/<0,500>[\QClsCloDist^{\op}`\QClsCloRel^{\op};(-)_0^{\op}]
\morphism|l|/@{->}@<-5pt>/<0,-600>[\QCatCls`\QSup;\sC]
\morphism(0,-600)/@{->}@<-5pt>/<0,600>[\QSup`\QCatCls;\sI]
\place(0,-300)[\mbox{\scriptsize$\dv$}]
\morphism(1000,0)|l|/@{->}@<-4pt>/<-1000,-600>[(\QClsDist)^{\op}`\QSup;\Cd]
\morphism(0,-600)|r|/@{->}@<-2pt>/<1000,600>[\QSup`(\QClsDist)^{\op};\Id]
\morphism(2200,-50)|l|<-2020,-520>[`;\Ccl]
\morphism(200,-650)|r|<2100,540>[`;\Icl]
\place(1200,-350)[\mbox{\rotatebox{15}{\scriptsize$\simeq$}}]
\morphism(-800,500)|a|/@{<-}@<4pt>/<800,0>[\QCat`\QCls;\sS]
\morphism(0,500)|b|/@{<-}@<4pt>/<-800,0>[\QCls`\QCat;\sD]
\place(-400,505)[\mbox{\rotatebox{90}{\scriptsize$\dv$}}]
\morphism(-880,450)|b|<720,-1020>[`;\sP]
\morphism(-80,-500)/^(->/<-700,1000>[`;]
\place(-460,-60)[\mbox{\rotatebox{35}{\scriptsize$\dv$}}]
\efig$$

Besides the adjunctions and equivalences illustrated in the above diagram, we also conclude the total (co)completeness of $\QCatCls$ and $\QCls$ through their topologicity respectively over $\QCat$ and $\Set\da\CQ_0$, and the (co)completeness of $\QClsCloDist$ and $\QClsCloRel$ through their monadicity over $\Set\da\CQ_0$. However, although $\QClsDist$ and $\QClsRel$ are respectively topological over $\QDist$ and $\QRel$, there is not much to say about the (co)completeness of $\QClsDist$ and $\QClsRel$, since $\QDist$ and $\QRel$ have few (co)limits as already the case $\CQ={\bf 2}$ shows.

\section*{Acknowledgements}

The author acknowledges the support of Natural Sciences and Engineering Research Council of Canada (Discovery Grant 501260 held by Professor Walter Tholen). The author is also grateful to the anonymous reviewers whose valuable remarks and suggestions significantly improve the presentation and readability of this paper.





\end{document}